\documentclass[12pt]{amsart}
\usepackage{etex}
\usepackage{graphicx}
\usepackage{epstopdf}
\usepackage{subcaption}

\usepackage{cite}
\usepackage{tikz, tikz-cd, tkz-graph}
\usetikzlibrary{arrows, patterns}

\usepackage[margin=1in]{geometry}
\usepackage{times}

\usepackage{xcolor}

\usepackage{hyperref}
\hypersetup{
  colorlinks,
  linkcolor={red!50!black},
  citecolor={blue!50!black},
  urlcolor={blue!80!black}
}

\usepackage[english]{babel}
\usepackage{amsmath,amssymb,amsthm}
\usepackage{xypic}
\usepackage{longtable}
\usepackage{array}

 \usepackage[enableskew]{youngtab}

\usepackage{epsfig}
\usepackage{hyperref}
\usepackage{enumitem}
\usepackage{booktabs}
 \usepackage{longtable} 
 \usepackage{pdflscape} 
 \usepackage{colortbl} 
 \usepackage{arydshln} 
 \usepackage{calc} 
\usepackage[vcentermath]{genyoungtabtikz}
\Yboxdim{6pt}

\graphicspath{{./images/}}

\newtheorem{thm}{\bf Theorem}[section]
\newtheorem{eg}[thm]{\bf Example}
\newtheorem{prop}[thm]{\bf Proposition}

\newtheorem{mydef}[thm]{\bf Definition}
\newtheorem{lem}[thm]{\bf Lemma}

\newtheorem{theorem}{\bf Theorem}

\newtheorem{theoremrepeat}{\bf Theorem}

\theoremstyle{remark}
\newtheorem{rem}[thm]{\bf Remark}

\newcommand{\CC}{\mathbb{C}}

\newcommand{\ZZ}{\mathbb{Z}}

\newcommand{\U}{\"u}
\newcommand{\flag}{\Fl(n;\mathbf{r})}
\newcommand{\QH}{\mathrm{QH}}
\newcommand{\HH}{\mathrm{H}}

\newcommand{\Sn}{S(n;\mathbf{r})}

\newcommand{\rr}{\mathbf{r}}
\newcommand{\ev}{\mathrm{ev}}   

\newcommand{\qq}{\mathbf{q}}   
\newcommand{\dd}{\mathbf{d}}   
\newcommand{\Sch}{\mathfrak{S}} 

\DeclareMathOperator{\sgn}{sgn}

\DeclareMathOperator{\Fl}{\mathrm{Fl}}

\DeclareMathOperator{\Gr}{\mathrm{Gr}}

\DeclareMathOperator{\rk}{\mathrm{rank}}

\renewcommand{\emptyset}{\varnothing}

\newcolumntype{C}[1]{>{\centering\let\newline\\\arraybackslash\hspace{0pt}}m{#1}}

\title[]{Quantum hooks and mirror symmetry for flag varieties}

\author[L.~Chen]{L.~Chen}
\address{Linda Chen \newline \indent Department of Mathematics and Statistics, Swarthmore College, Swarthmore, PA 19081}
\email{lchen@swarthmore.edu}

\author[E.~Kalashnikov]{E.~Kalashnikov}
\address{Elana Kalashnikov \newline \indent  Department of Pure Mathematics,University of Waterloo,
Waterloo, Canada N2L 3G1}
\email{e2kalash@uwaterloo.ca}

\thanks{LC was partially supported by Simons Collaboration Grant 524354 and NSF Grant DMS-2101861. EK is supported by an NSERC Discovery Grant.}

\begin{document}
\begin{abstract}
Given a flag variety $\Fl(n;r_1, \dots , r_\rho)$, there is natural ring morphism from the symmetric polynomial ring in $r_1$ variables to the quantum cohomology of the flag variety. 
In this paper, we show that for a large class of partitions $\lambda$, the image of $s_\lambda$ under the ring homomorphism is a Schubert class which is described by partitioning $\lambda$ into a quantum hook (or $q$-hook) and a tuple of smaller partitions.  We use this result to show that the Pl\U cker coordinate mirror of the flag variety describes quantum cohomology relations. This gives new insight into the structure of this superpotential, and the relation between superpotentials of flag varieties and those of Grassmannians (where the superpotential was introduced by Marsh--Rietsch). 
\end{abstract}

\maketitle

\section{Introduction}\label{sec:intro}
The extension of mirror symmetry for Fano varieties beyond the toric context, where it is well-understood due to foundational work by Hori--Vafa \cite{horivafa}, Givental \cite{Givental98}, Lian--Liu--Yau \cite{lian} and others, is an area of active research. Grassmannians and flag varieties are central examples here, as Fano GIT quotients with a rich geometric and combinatorial structure.  

One of the oldest proposals for a mirror, or superpotential, for the Grassmannian $\Gr(n,r)$ was given by Eguchi--Hori--Xiong \cite{eguchi}. This was later generalized to type A flag varieties by  Batyrev--Ciocan-Fontanine--Kim--van Straten \cite{flagdegenerations} (for simplicity, we refer to these mirrors as EHX mirrors).  These proposals are motivated by taking toric degenerations of Grassmannians and flag varieties, and then applying toric methods to the singular fiber. However, the toric degeneration approach has not been successful in proving required properties of these mirrors -- partial verification was completed for Grassmannian and flag varieties by Rietsch in \cite{lietheoretic} using the Lie theoretic superpotential, and a full verification for Grassmannians by Marsh--Rietsch \cite{MarshRietsch} using the Pl\U cker coordinate mirror. 

The Pl\U cker coordinate mirror for the Grassmannian is the most promising approach to mirror symmetry beyond the toric context. This remarkable construction connects the earlier, Lie theoretic proposals of the Grassmannian with the conjectures of the Fanosearch program and the toric degeneration approach.  Extending the construction beyond the Grassmannian is thus an important problem.  In \cite{kflags}, the second author introduces a conjectural Pl\U cker coordinate mirror for type A flag varieties (see \cite{spacek, spacekwang} for recent progress on the subject in other types). As a first test of its validity, the second author proves in \cite{kflags} that the Pl\U cker coordinate mirror is compatible with the EHX mirror. More is required, however: a superpotential or mirror should compute quantum information about the variety -- through determining both quantum relations as well as certain genus 0 Gromov--Witten invariants. 

In this paper, we prove a theorem in this direction. We show that partial derivatives of the Pl\U cker coordinate mirror of a type A flag variety give quantum cohomology relations.  To state the result carefully, we need some more background. 

The Pl\U cker coordinate mirror of the Grassmannian is a rational function on the Grassmannian. As for toric varieties, there is a map from the Cox ring of the Grassmannian (the ring generated by Pl\U cker coordinates) to the cohomology ring of the Grassmannian.  Pl\U cker coordinates of the Grassmannian of quotients $\Gr(n,r)$ are indexed by the same set as Schubert classes of the Grassmannian -- i.e. by partitions fitting into an $r \times (n-r)$ box -- and this map takes the Pl\U cker coordinate $p_\lambda$ to the Schubert class $s_\lambda$. Under this map, partial derivatives of the Pl\U cker coordinate mirror give quantum cohomology relations \cite{MarshRietsch}. 

For $n=:r_0$ and $\mathbf{r}=(r_1>\cdots>r_\rho>r_{\rho+1}:=0)$, let $\flag:=\Fl(n;r_1,\dots,r_\rho)$ be the partial flag variety of successive quotients of $\CC^n$ of dimension $r_i$. The Pl\U cker coordinate mirror of this flag variety proposed in \cite{kflags}  is a rational function on a product of Grassmannians $Y=\prod_{i=1}^\rho \Gr(r_{i-1},r_i)$, with the convention $r_0:=n$. Using the cluster structures of the Grassmannian factors, this superpotential can be written as a Laurent polynomial in certain Pl\U cker coordinates of each factor. We index Pl\U cker coordinates on $Y$ by $p^i_\lambda$, where $i=1,\dots,\rho$ and $\lambda$ is a partition that fits into an $r_i \times (r_{i-1}-r_i)$ box.  

 Schubert classes of the Grassmannian are indexed by partitions, and Schubert classes $\sigma_{\vec{\lambda}}$ in a flag variety are indexed by  \emph{tuples} ${\vec{\lambda}}=(\lambda_1,\dots,\lambda_n)$ of partitions.  To interpret partial derivatives of the Pl\U cker coordinate mirror requires a map from the Cox ring of $Y$ to the cohomology of the  flag variety. This is not the natural map given by 
 \[p^i_\lambda \mapsto s_{\vec{\mu}},\]
 where $\vec{\mu}_j$ is $\lambda$ if $i=j$ and $\emptyset$ otherwise. Instead, we require the \emph{Schubert map}  $F$ (see definition \ref{def:Schubertmap}).  Our main result is then the following. 

\begin{theorem}  \label{thm:thmA} Let $W_P$ be the Pl\U cker coordinate mirror of a flag variety, and $W_{P,C}$ the expression of $W_P$ in any choice of cluster charts.  Then
\[F\left(\frac{\partial}{\partial p^i_{\lambda}} W_{P,C}\right)=0\]
in quantum cohomology, for any $i=1,\dots,\rho$ and $p^i_\lambda$ in the cluster chart. 
\end{theorem} 

This result represents a significant step towards a full verification of the Pl\U cker coordinate mirror of flag varieties. If similar structure holds beyond the type A case, this result may also be important in extending candidate mirrors from cominiscule varieties to any homogeneous space.  It elucidates the increased complexity from the Grassmannian case. It also demonstrates a previously unobserved structure relating the mirrors of Grassmannians and flag varieties: although not at all obvious from the description of the Schubert map, we show that it precisely interpolates between the Pl\U cker coordinate mirror of the flag variety $\flag$ and containing Grassmannians $\Gr(N,r_1), N>>0.$ It is this property of the Schubert map which is key to the proof of Theorem A, as it essentially allows us to reduce to the Grassmannian case. 
This interpolation result is a corollary of Theorem B below, a purely quantum cohomology statement. 

Following the approach of \cite{fgp} and \cite{cf}, we use a ``quantization" approach for the quantum cohomology ring of the flag variety. This and other descriptions will be reviewed in Sections \ref{sec:background} and \ref{sec:background2}. There is a natural ring homomorphism from  the ring $\Lambda_{r_1}$ of symmetric polynomials in $r_1$ variables to $\QH^*\flag$ defined sending elementary symmetric polynomials to certain quantum elementary polynomials. We write $s^1_\lambda \in \QH^*\flag$ for the image of a Schur polynomial $s_\lambda$ (see \eqref{eq:si} for more details). The first part of Theorem B states that for certain partitions $\lambda$, the image is  a  Schubert class (up to multiplication by quantum parameters), and the second part of Theorem B states that for another class of  partitions, the image is zero.

For $0<b\leq n$, let $0\leq I\leq \rho$ be such that $n-r_I<b\leq n-r_{I+1}$. In \S \ref{sec:theoremB},  we define the \emph{quantum hook} or \emph{$q$-hook} of width $b$ to be the partition  $H_b := (b^{b-n+r_1},(b-n+r_I)^{n-r_{I+1}-b})$, and  set $R_b:=(b^{b-n+r_1})$ to be the maximal width rectangle contained in $H_b$,  with  $H_b=R_b=\emptyset$ if $b<n-r_1$. Set $$q^{H_{b}}:=  q_1^{r_1-r_2}\cdots(q_1\cdots q_{I-1})^{r_{I-1}-r_I}(q_1\cdots q_I)^{b-(n-r_I)}.$$  
For a partition $\lambda$ that contains the $q$-hook of  width equal to the width of $\lambda$,  we associate a tuple of partitions $\vec{\mu}=(\mu^1,\ldots,\mu^{I+1},\emptyset,\ldots,\emptyset)$ by subdividing the skew shape $\lambda/H_{\lambda}$ as in Figure \ref{fig:mu}, where $\mu_i\in P(r_{i-1},r_i)$ is of width $r_{i-1}-r_i$. (See Definition \ref{def:mu-lambda} for more details.)

\begin{figure}[h!]
\centering
 
\begin{tikzpicture}[scale=.4]

\fill[color=gray!40]   (0,0) --(9, 0) -- (9,-5) -- (2, -5) -- (2, -7)-- (0, -7) -- cycle;

\draw[thick] (6, 0) -- (9, 0) -- (9,-5) -- (2, -5) -- (2, -7) -- (0, -7)--(0,-3) -- (0,0) -- (6,0);

\draw[thick] (0,-7) --  (0,-9.67) -- (1.33,-9.67)-- (1.33,-9) -- (2,-9)-- (3.33,-9) -- (3.33,-8)  -- (6.5,-8)--(6.5,-7)--(8.2,-7) --  (8.2,-6)--(9,-6) --(9,-5);

\draw (2,-6) -- (2,-9);
\draw (4.67,-5) -- (4.67,-8);
\node[scale=.8] at (5, -6.1) {$\cdots$};
\draw (5.3,-5) -- (5.3,-8);
\draw (7,-5) -- (7,-7);
\node[scale=.8]  at (1, -8) {$\mu^{I+1}$};

\node[scale=.8]  at (3.3,-6.1) {$\mu^I$};

\node[scale=.8] at (6.1,-6.1) {$\mu^2$};
\node[scale=.8] at (7.6,-6.1) {$\mu^1$};

\node[scale=.8] at (4.5,-3) {$H_{\lambda_1}$};

\end{tikzpicture}
\caption{A partition $\lambda$ containing $H_{\lambda_1}$,  the skew shape $\lambda/H_{\lambda_1}$, and the associated tuple of partitions $\vec{\mu}=(\mu^1,\ldots,\mu^{I+1},\emptyset,\ldots,\emptyset)$.  }
\label{fig:mu}

\end{figure}
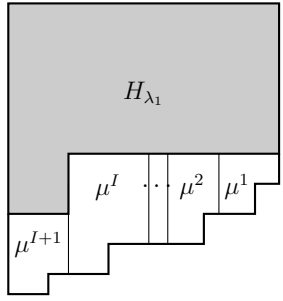

\begin{theorem}\label{thm:thmB} 
Let $\lambda\subseteq r_1\times n$ be a partition, and let $I$ be such that $n-r_I<\lambda_1\leq n-r_{I+1}$. 
\begin{enumerate}
\item[(a)]  If $H_{\lambda_1}\subseteq \lambda$, then
\[ s^1_\lambda = q^{H_{\lambda_1}}\sigma_{\vec{\mu}} \text{ in } \QH^*\flag,
\]
 where $\vec{\mu}=(\mu^1,\ldots,\mu^{I+1},\emptyset,\ldots,\emptyset)$ is the tuple of partitions associated to $\lambda$ above.
 \item[(b)] If $\lambda$ contains $R_{\lambda_1}$, but $H_{\lambda_1}\not\subseteq \lambda$, then  
$$s^1_\lambda =0 \text{ in } \QH^*\flag.$$
 \end{enumerate}
  In particular,  $s^1_{H_{\lambda_1}} = q^{H_{\lambda_1}}$ since  $H_{\lambda_1}/H_{\lambda_1}=\emptyset$, so  $\mu^j=\emptyset$ for all $j$ and $\sigma_{(\emptyset,\ldots,\emptyset)} = 1$.
  \addtocounter{theorem}{-1}
\end{theorem}

In \S \ref{sec:background}, we review the necessary background on quantum cohomology of Grassmannians and flag varieties, and in \S \ref{sec:background2}, we discuss the EHX and Pl\U cker coordinate mirror of the Grassmannian. In \S \ref{sec:theoremA}, we describe the Schubert map and prove Theorem A, and in \S \ref{sec:theoremB}, we study  $q$-hooks and prove Theorem B. 
\subsection*{Acknowledgements}
The authors would like to thank Konstanze Rietsch, Dave Anderson, and Jennifer Morse for helpful conversations. 
\section{Quantum cohomology of flag varieties}\label{sec:background}
\subsection{Permutations and Schubert classes}
Fix an $n$-dimensional vector space $V$ and a tuple of integers $\mathbf{r}= (n>r_1>\cdots > r_\rho>0)$ and let 
$\flag= \Fl(n;r_1,\dots,r_\rho) $ denote the partial flag variety parametrizing successive quotient flags of $V$ of dimensions $r_i$. It comes equipped with a tautological sequence of quotient bundles 
$V_{\flag} \twoheadrightarrow Q_1 \twoheadrightarrow \cdots \twoheadrightarrow Q_\rho$ 
of ranks $r_1,\dots,r_\rho$. 

The basis of Schubert classes for $\flag$ consists of geometrically described cohomology classes that are commonly indexed by
\[ S(n;\mathbf{r}):= \{w\in S_n: w(i)<w(i+1) \text{ if } i\not\in\mathbf{r}\},\]
 the set of permutations in $S_n$ whose descent set is contained in $\{r_1,\ldots,r_\rho\}$.
If 
$S_{n,\mathbf{r}}$ is the  parabolic subgroup of $S_n$ generated by simple transpositions $(i,I+1)$ for $i\not\in\rr$, then $\Sn$ is a set of coset representations for $S_n/S_{n,\rr}$.

For a permutation $w\in S(n,\mathbf{r})$, define $r_w(p,q) = \#\{i \leq p \,|\, w(i)\leq q\}.$ 
This is the rank of the upper-left $p\times q$ submatrix of the permutation matrix corresponding to $w$ (which has $1$'s in positions $(i,w(i))$ and $0$'s elsewhere).   The \emph{length} of $w$ is the number
$\ell(w) = \#\{ i<j \,|\, w(i)>w(j) \}$.

Let $E_\bullet$ be a flag of trivial vector bundles on $\flag$. Then the Schubert variety $\Omega_w = \Omega_w(E_\bullet) \subseteq \flag$ is defined by
\[
  \Omega_w = \{ x\in \flag) \,|\, \rk(E_q \to Q_p) \leq r_w(n_p,q) \text{ for all } 1\leq q\leq n, n_p\in\mathbf{r}\}.
\]
We write  $\sigma_w$  for the corresponding  Schubert class in $\HH^{2\ell(w)}\flag$.

The unique permutation of longest length in $\Sn$ is given explicitly by 
\[
  w^\circ = [ n-r_\rho+1,\ldots,n,\; \ldots, n-n_{1}+1, \ldots, n-r_2, \ldots, 1,2,\ldots,n-r_1].
\]
Its length is $\ell(w^\circ) = \dim\flag$. There is an involution on $\Sn$, using the longest element $w_\circ^{\mathbf{r}}$ of $S_{n,\mathbf{r}}$:
\[
  w^\vee = w_\circ \cdot w \cdot w_\circ^\rr.
\]
This is an element of $\Sn$, with $\ell(w^\vee) = \dim\flag - \ell(w)$.  The classes $\sigma_{w^\vee}$ form a Poincar\'e dual basis: $\int_{\flag} \sigma_w \cup \sigma_{v^\vee}  = \delta_{w,v}$.

\subsection{Another basis and tuples of partitions}

We describe another basis for $\HH^*\flag$ in terms of tuples of partitions. Consider the set
\[P(n,\mathbf{r}):=\prod_{i=1}^\rho P(r_{i-1},r_{i}),\]
where we  set $r_0:=n$ and $r_{\rho+1}:=0$, and where
  $P(a,b)$ denotes the partitions inside  a $b \times (a-b)$ rectangle.

\begin{rem} \label{rem:bijection} There is a bijection between permutations in $S(n,\mathbf{r})$ and tuples of partitions in $P(n,\mathbf{r})$. Given a tuple $\vec{\mu}=(\mu^1,\ldots,\mu^\rho)\in P(n,\mathbf{r})$, for $1\leq i\leq \rho$, denote by $w^i$  the Grassmannian permutation in $S_n$ with possible descent at $r_i$ defined by the partition $\mu^i\subseteq P(  r_{i-1},r_i)\subseteq P(n,r_i)$, i.e. $\mu^i = (w^i(r_i)-r_i,\ldots w^i(1)-1)$, so that $w^i = w_{(\emptyset,\ldots, \mu^i, \ldots,\emptyset)}$  Then the  tuple $\vec{\mu} \in P(n,\mathbf{r})$ corresponds to the permutation
\[ w_{\vec{\mu}} := w_{(\mu^1,\emptyset,\ldots)} \cdots w_{(\emptyset,\ldots,\mu^\rho)} = w^1w^2\cdots w^\rho.\]
On the other hand, given $w\in S(n,\mathbf{r})$, we can produce a tuple $\vec{\mu}=(\mu^1,\ldots,\mu^\rho)$. 
(See  also \cite{WiAG}).
\end{rem}

 If  a tuple of partitions $\vec{\mu}=(\mu^1,\ldots,\mu^\rho)\in P(n,\mathbf{r})$ corresponds to the permutation $w$ under the bijection in Remark \ref{rem:bijection}, we also write the Schubert class $\sigma_w$ as $\sigma_{\vec{\mu}}$.

\begin{eg} \label{eg:permutation} Consider the flag variety $\Fl(8;6,4,3)$ with $n=8$ and $\mathbf{r}=(6,4,3)$. For the  tuple  $\left(\yng(2,1),\yng(1,1),\emptyset\right)$ in $P(n,\mathbf{r})$,  $w^1=[123457|68]\, w^2=[1245|3678], 
w^3=id$ with descents marked at $r_1=6$ and $r_2=4$. The corresponding permutation in $S(n,\mathbf{r})$ is $w=w^1w^2w^3=[1245|37|68]$. 
Similarly, $\left( \yng(2,1),\yng(2,2,1),\yng(1,1,1)\right)$ corresponds to the permutation
   $ [123468|57]\cdot [1356|2478]\cdot
[234|15678]=[368|1|24|57]$, and  $\left(\yng(1),\yng(2,1),\yng(1,1)\right)$ corresponds to 
$[123457|68]\cdot [1246|3578]\cdot[134|25678]=[147|2|35|68]$.

\end{eg}
 For a partition $\lambda\in P(r_{i-1},r_i)$, we define the class $s^i_\lambda$ to be the Schur polynomial associated to the partition $\lambda$ in the Chern roots of $Q_i$, the rank $r_i$ tautological quotient bundle on $\flag$:
\begin{equation}
\label{eq:s-def}
s^i_\lambda=\det(s^i_{1^{\lambda'_k+l-k}}).
\end{equation}
Note that $s^i_{1^a}=c_a(Q_i)$ is the $a$th Chern class of the bundle $Q_i$ so that $s^i_{1^a}=e_a(r_i)$.  Via the bijection in Remark \ref{rem:bijection}, $s^i_\lambda$ is equal to the Schubert class 
$$s^i_\lambda= \sigma_{(\emptyset,\ldots,\lambda,\ldots,\emptyset)}$$ 
associated to the tuple of partitions consisting $i$th partition equal to $\lambda$ and the empty partition elsewhere.

Note that  we can use \eqref{eq:s-def} to define $s^i_\lambda$ even when $\lambda \not \in P(r_{i-1},r_i)$, although it is no longer  a Schubert class in general.

\begin{rem}
Given a tuple of partitions $\vec{\mu}=(\mu^1,\ldots,\mu^\rho)$, we obtain another important  class
$$s_{\vec{\mu}}:= s^1_{\mu_1} \cdots s^\rho_{\mu_\rho}.$$
Running over all $\vec{\mu}$ we obtain another basis for the cohomology of the flag variety.
 The two bases $\{\sigma_{\vec{\mu}}\}$ and $\{s_{\vec{\mu}}\}$ are distinct, except in the case of the Grassmannian. \end{rem}

\subsection{Quantum cohomology}

The \emph{ quantum cohomology ring} $\QH^*\flag$ is a commutative and associative graded algebra over $\ZZ[q_1,\ldots,q_\rho]$, where $q_i$ is a parameter of degree $r_{i-1}-r_{i+1}$.  As a module, $\QH^*\flag$
 is simply $\ZZ[q]\otimes_\ZZ H^*\flag$, so it has a $\ZZ[q]$-basis of Schubert classes $\sigma_w$:
\[
  \QH^*\flag = \bigoplus_{w\in \Sn} \ZZ[q]\cdot \sigma_w.
\]
The  quantum product is a deformation of the usual  product.  For permutations $u,v\in \Sn$, define a product by
\[
  \sigma_u * \sigma_v = \sum_{w,\dd} \qq^\dd\, c_{u,v}^{w,\dd}\, \sigma_w,
\]
where $\dd$ ranges over $(n-1)$-tuples of nonnegative integers, and the \emph{three-pointed  Gromov-Witten invariant} $c_{u,v}^{w,\dd}$ is defined as follows.

Let $\overline{M}_{0,3}(\flag,\dd)$ be the Kontsevich moduli space of three-pointed genus-zero stable maps to $\flag$ of degree $\dd$,  parametrizing data  $(f,C,(x_1,x_2,x_3))$, where $C$ is a genus-zero curve with marked points $x_i$,  $f:C\to\flag$ is a map of degree $\dd$, and a certain stability condition is imposed \cite{Kontsevich95}.  The space of stable maps is of dimension  $\dim \flag + \sum_{i=1}^3 d_i(r_{i-1}-r_{i+1})$,
and comes with natural evaluation morphisms
\[\ev_i: \overline{M}_{0,3}(\flag,\dd)\to \flag
\]
for $1\leq i\leq 3$ that send $(f,C,(x_1,x_2,x_3))$ to $f(x_i)$.  Now one defines
$c_{u,v}^{w,\dd}= \pi_*( \ev_1^*\sigma_u \cdot \ev_2^*\sigma_v \cdot \ev_3^*\sigma_{w^\vee} ).$
This defines an associative product. See \cite{fp} for more details on quantum cohomology.

\subsection{Quantum cohomology of flag varieties} \label{sec:qh}

The Schubert polynomials of Lascoux and Sch\"utzenberger are defined inductively, starting from $\Sch_{w_\circ}(x) = x_1^{n-1} x_2^{n-2} \cdots x_{n-1}$ and moving down Bruhat order using divided difference operators   \cite{ls}.  For any $w\in S_n$, the polynomial $\Sch_w(x)$ has a unique expansion in terms of elementary symmetric polynomials:
\begin{equation}\label{e:sch-elem}
  \Sch_w(x) = \sum a_{k_1\ldots k_{n-1}}\,e_{k_1}(1) \cdots e_{k_{n-1}}(n-1)
\end{equation}
over sequences $(k_1,\ldots,k_{n-1})$ with $0\leq k_j\leq j$ and $\sum k_j=\ell(w)$, where the $a_{k_1\ldots k_{n-1}}$ are integers and $e_k(j):=e_k(x_1,\ldots,x_j)$ is the $k$th elementary symmetric polynomial in the variables $x_1,\ldots,x_l$.

Let
\[
  \sigma_1^\rho,\ldots,\sigma_{r_\rho}^\rho,\,\sigma_1^{\rho-1},\ldots,\sigma_{r_{\rho-1}-r_\rho}^{\rho-1},\ldots,\sigma_1^{0},\ldots,\sigma_{n-r_1}^{0}
\] 
be $n$ independent variables, with $\sigma_i^j$ of degree $i$.  To form  quantum polynomials for $\flag$, one replaces $e_k(j)$ with {\em quantum elementary polynomials} $e^q_k(r_l)$, which are defined for $r_l \in\rr$ and $r_0=n$ recursively by   
 \begin{equation} \label{eq:eq-recursion}
 e^{\rr,q}_a(r_{l-1}) = \sum_{m=0}^{r_{l-1}-r_l} \sigma_m^l e^{\rr,q}_{a-m}(r_l) + (-1)^{r_{l-1}-r_l+1}q_l \, e^{\rr,q}_{a-(r_{l-1}-r_{l+1})}(r_{l+1}),
 \end{equation}
 where we set $e^q_0(r_l)=1$ and $e^q_m(r_l)=0$ if either $m<0$ or $m>r_l$. When $\rr$ is understood, we simply write $e^q_k(r_l)$ for $e^{\rr,q}_k(r_l)$. (Our conventions here differ from those found elsewhere in the literature, e.g. our $r_l$ and $\sigma_i^j$ correspond to $n_{\rho+1-l}$ and   $\sigma_i^{\rho+1-j}$ in \cite{cf}.)

From  \cite{cf,kimpresentation}, we know a presentation of the quantum cohomology ring and polynomial representatives of the quantum Schubert classes. 
\[
\QH^*\flag \cong \ZZ[q][\  \sigma_1^\rho,\ldots,\sigma_{r_\rho}^\rho,\ldots,\sigma_1^{0},\ldots,\sigma_{n-r_1}^{0}]/I^q,
\]
where $I^q$ is the ideal $({e}_1^{\rr,q}(n),\ldots,{e}_n^{\rr,q}(n))$ generated by $n$ relations ${e}_1^{\rr,q}(r_0)=0$ which specialize to the known relations defining $H^*\flag$ when $q \mapsto 0$, and
\[ \sigma_w =  \Sch^{\rr,q}_w(\sigma)\]
for $w\in \Sn$, where the {\em quantum Schubert polynomial} $\Sch^{\rr,q}_w(\sigma)$ is formed by substituting $e^{\rr,q}_k(r_l)$ for $e_k(j)$ on the RHS of \eqref{e:sch-elem} whenever $j\in[r_l,r_{l-1})$.   

The quantum structure constants of the alternate basis, $s_{\vec{\mu}}$, can be computed using rim-hook removals via the Abelian/non-Abelian correspondence \cite{gukalashnikov}. 

\subsection{Determinantal formulas}

In Section \ref{sec:theoremB}, we will study certain skew shapes $\lambda/\mu$ along with a labeling $\omega(i,j)=r_1+i-j$. 
By \cite{bjs},  associated to  $(\lambda/\mu,\omega)$ is a 321-avoiding permutation $w$ whose corresponding Schubert polynomial  is equal to a \emph{flagged skew Schur polynomial} that can be expressed as a determinant:
 \begin{equation}\label{eq:schub-det}
 \Sch_w(x) =  \left| e_{\lambda'_i-\mu'_j+j-i}(f_j) \right|_{1\leq i,j\leq t}
 \end{equation}
where $f_j=\omega(j,\lambda'_j)=r_1+j-\lambda'_j$ is the ``flagging" associated to $w$. 

 For a skew shape $\lambda/\mu$ and $\phi=(\phi_1,\ldots,\phi_t)$ with $1\leq \phi_i \leq \rho$, define
  \begin{equation} \label{eq:q-determinant}
\Delta_{\lambda/\mu}(e^q(\phi)):=\left| e^q_{\lambda'_i-\mu'_j+j-i}(r_{\phi_j}) \right|_{1\leq i,j\leq t}.
\end{equation}
When $\phi_j$ is   defined by $r_{\phi_j} \leq f_j <r_{\phi_j -1}$, substituting $e_k(j)=e^q_k(r_l)$  in \eqref{eq:schub-det}  as in  the discussion in Section \ref{sec:qh}, we obtain a determinantal expression for the
quantum Schubert class:
 \begin{equation} \label{eq:q-Schubert}\sigma_w = \Delta_{\lambda/\mu}(e^q(\phi)) \text{ in }QH^*\flag.
 \end{equation}

We can also define quantum classes $s^i_\lambda$ for partitions $\lambda$ by computing the determinant
 \eqref{eq:s-def} using the quantum product. 
When $\lambda\in P(r_{i-1},r_i)$, this gives the quantum Schubert class $\sigma_{\emptyset,\cdots,\lambda,\cdots,\emptyset}$, but $s^i_\lambda$ is also defined  when $\lambda\not\in P(r_{i-1},r_i)$. In particular, since $s^i_{1^a}=e_a(r_i)$ classically,  we have $s^i_{1^a} = e^q_a(r_i)$  in $\QH(\flag)$ and 
\begin{equation}\label{eq:si}
s^i_\lambda= \left| s^i_{1^{\lambda'_k+l-k}} \right|_{1\leq k,l\leq \lambda_1} =\Delta_\lambda(e^q(\phi)),
\end{equation}
where $\phi=(i,\ldots,i)$.

\begin{rem} \label{rem:skewzero}
If $\mu\not\subseteq\lambda$, then $\lambda'_k<\mu'_k$ for some $1\leq k\leq t$. If $i\geq k$ and $j\leq k$, the $(i,j)$th entry of the matrix in $\Delta_{\lambda/\mu}$ is indexed by $\lambda'_i-\mu'_j+j-i<\lambda'_k-\mu'_k<0$, and so is zero. Since the matrix is block upper triangular with left upper block of determinant zero, $\Delta_{\lambda/\mu} =0$.
\end{rem}

\section{Cluster structure and superpotentials}\label{sec:background2}

\subsection{The cluster structure of the Grassmannian}
In this section, some brief facts about the cluster structure of the Grassmannian are recalled. Good references include \cite{scott, RW}. Fix a Grassmannian of quotients $\Gr(n,r)$. Pl\U cker coordinates on the Grassmannian are indexed by partitions $\lambda$ fitting in an $r \times (n-r)$ box, i.e. by $\lambda \in P(n,r)$. The homogeneous coordinate ring of the Grassmannian is generated by $p_\lambda, \lambda \in P(n,r)$, and relations are given by the Pl\U cker relations. 

This ring, as well as certain localizations of it, has a cluster structure. Certain sets of algebraically independent Pl\U cker coordinates are clusters. An important example of a cluster is the \emph{rectangles cluster}.

 \begin{mydef} The \emph{rectangles cluster chart} is the set of Pl\U cker coordinates indexed by all partitions $\lambda \in P(n,r)$ such that $\lambda$ is a rectangle. 
 \end{mydef}
 
 One cluster can be obtained from another via \emph{mutation}. These mutations arise from three-term Pl\U cker relations \cite{scott}. The three term quadratic Pl\U cker relations  are of the form
\[p_\lambda p_\mu = p_a p_b + p_c p_d\]
where $\lambda, \mu, a, b, c, d \in P(n,r)$ are six partitions related in a particular way. A cluster containing $p_\lambda, p_a, p_b, p_c,$ and $p_d$ can be mutated to one containing $p_\mu,  p_a, p_b, p_c,$ and $p_d$.  Any cluster is related to any other by a series of mutations of this form. 

\subsection{Superpotentials of the Grassmannian}
The Eguchi--Hori--Xiong (EHX) superpotential  for Grassmannians is described by building a ladder diagram for the Grassmannian,  and super-imposing a dual quiver on the diagram. The ladder diagram for the Grassmannian $\Gr(n,r)$ is an $r \times (n-r)$ grid. There is a toric degeneration of the Grassmannian to the quiver moduli space described by a quiver originating from the ladder diagram.  The superpotential is given by a head-over-tails process on the dual quiver. We illustrate this briefly in the example $\Gr(5,2)$: the ladder diagram is a $2 \times 3$ grid:
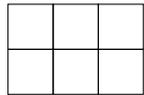
\[\begin{tikzpicture}[scale=0.6]
\draw (0,0) rectangle (1,1);
\draw (1,1) rectangle (2,2);
\draw (0,1) rectangle (1,2);
\draw (1,0) rectangle (2,1);
\draw (2,1) rectangle (3,2);
\draw (2,0) rectangle (3,1);
\end{tikzpicture}.\]
The dual quiver is then:
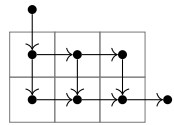
\[\begin{tikzpicture}[scale=0.6]
\draw[gray] (0,0) rectangle (1,1);
\draw[gray] (1,1) rectangle (2,2);
\draw[gray] (0,1) rectangle (1,2);
\draw[gray] (1,0) rectangle (2,1);
\draw[gray] (2,1) rectangle (3,2);
\draw[gray] (2,0) rectangle (3,1);
\node[circle,fill,scale=0.3] at (0.5,0.5) (a) {};
\node[circle,fill,scale=0.3] at (1.5,0.5) (b) {};
\node[circle,fill,scale=0.3] at (2.5,0.5) (c) {};
\node[circle,fill,scale=0.3] at (0.5,1.5) (e) {};
\node[circle,fill,scale=0.3] at (0.5,2.5) (f) {};
\node[circle,fill,scale=0.3] at (1.5,1.5) (i) {};
\node[circle,fill,scale=0.3] at (2.5,1.5) (k) {};
\node[circle,fill,scale=0.3] at (3.5,0.5) (m) {};
\draw[<-] (c)--(b);
\draw[<-] (b)--(a);
\draw[<-] (m)--(c);
\draw[<-] (k)--(i);
\draw[<-] (i)--(e);

\draw[->] (e)--(a);
\draw[->] (f)--(e);
\draw[->] (i)--(b);
\draw[->] (k)--(c);
\end{tikzpicture}.\]
In general, to form the dual quiver, place a vertex in each box of the ladder diagram, as well as one at the top left and bottom right of the diagram, and then add arrows oriented down and right. 

To obtain the superpotential, assign to each of the vertices a variable $z_{ij}$, where $i$ indicates the row (starting at 0) and $j$ the column (starting at 1). We set $z_{10}=1$ and $z_{r (n-r+1)}=q$. The EHX superpotential is then:
\[W_{EHX}=\sum_{a} \frac{z_{h(a)}}{z_{t(a)}}.\]
 The sum is over the arrows of the quiver, and $h(a)$ and $t(a)$ indicate the head and tail of an arrow respectively. 
 \begin{eg} The EHX superpotential for $\Gr(4,2)$ is
 \[z_{11}+\frac{z_{12}}{z_{11}}+\frac{z_{21}}{z_{11}}+\frac{z_{22}}{z_{12}}+\frac{z_{22}}{z_{21}}+\frac{q}{z_{22}}.\]
 \end{eg}
  
A superpotential is a mirror to a Fano manifold if information about the genus 0 Gromov--Witten invariants of the Fano manifold can be computed by the superpotential. More precisely, one or both of the following conditions might hold:
\begin{enumerate}
\item The period sequence of the superpotential is equal to the regularized quantum period of the Fano manifold (see \cite{fanomanifolds} for definitions and details).
\item The Jacobi ring of the superpotential computes the quantum cohomology ring of the Fano manifold. 
\end{enumerate}
The first condition was the original conjecture of Eguchi--Hori--Xiong, later proved by Marsh--Rietsch \cite{MarshRietsch} for the Grassmannian.  This conjecture remains open for flag varieties.   

The second condition -- that the superpotential produces relations in the quantum cohomologry  ring -- is the central focus of the paper. We first discuss the proof in the case of the Pl\U cker coordinate mirror for the Grassmannian,  introduced by Marsh--Rietsch  in \cite{MarshRietsch}; the same statement for the EHX mirror is obtained as a corollary. 

To construct the Pl\U cker coordinate mirror the Grassmannian $\Gr(n,r)$, take $n$ equations of the form 
 \[s_{\yng(1)} * s_\lambda=q^i s_{\mu}\]
 where $i=0,1$ depending on the partition. Here $\lambda=(a,\dots,a) \in P(n,r)$ is either the empty set or a rectangular partition either maximally wide or maximally tall: we denote the set of such partitions $M(n,r)$. 
 \begin{rem} $M(n,r)$ is the set of \emph{frozen variables} in the cluster structure of the Grassmannian: they appear in every cluster. 
 \end{rem}
 Note that the sum 
 \[ \sum_{\lambda \in M(n,r)} \frac{q^{i} s_\mu}{s_\lambda}\]
is equal to $n s_{\yng(1)}=-K_{\Gr(n,r)}$, the anti-canonical class of the Grassmannian. An analogous statement is true for the Hori--Vafa mirror of a toric variety. 

To transform the sum into a (rational) function, every Schubert class $s_\lambda$ is replaced with the Pl\U cker coordinate $p_\lambda$.
\begin{eg} 
The Marsh--Rietsch Pl\U cker coordinate superpotential for $\Gr(4,2)$ is
\[\frac{p_{\yng(1)}}{p_{\emptyset}}+\frac{p_{\yng(2,1)}}{p_{\yng(2)}}+\frac{p_{\yng(2,1)}}{p_{\yng(1,1)}}+\frac{q p_{\yng(1)}}{p_{\yng(2,2)}}.\]
\end{eg}
Following \cite{MarshRietsch}, we denote the open subvariety on which the Pl\U cker coordinate superpotential is a function (i.e. where $p_\lambda \neq 0$, $\lambda \in M(n,r)$) as $\Gr(n,n-r)^\circ$. 

Using Pl\U cker relations, we can expand the Pl\U cker coordinate mirror into a Laurent polynomial in each cluster chart in the cluster structure on the coordinate ring of the Grassmannian. 
 \begin{eg} We can use the three term Pl\U cker relation
 \[p_{\yng(1)} p_{\yng(2,1)}=p_{\yng(1,1)}p_{\yng(2)}+p_{\emptyset} p_{\yng(2,2)}\]
 to find that in the rectangles cluster chart, the mirror for $\Gr(4,2)$ is
\[\frac{p_{\yng(1)}}{p_{\emptyset}}+\frac{p_{\yng(2)}}{p_{\yng(1)}}+\frac{p_{\yng(1,1)}}{p_{\yng(1)}}+\frac{p_{\yng(2,2)}}{p_{\yng(2)}}+\frac{p_{\yng(2,2)}}{p_{\yng(1,1)}}+\frac{q p_{\yng(1)}}{p_{\yng(2,2)}}.\]
 \end{eg}
In each cluster chart, one can compute the critical locus by setting the partial derivatives to zero: $\frac{\partial}{\partial W_C}W_C=0.$ It is clear how to interpret these equations as candidate relations in quantum cohomology: both Pl\U cker coordinates and Schubert classes of the Grassmannian $\Gr(n,r)$ are indexed by the same set of partitions, $\lambda \subset r \times (n-r)$. 
\begin{thm}[\cite{MarshRietsch}]
The Jacobi ring of the Pl\U cker coordinate mirror is isomorphic to the quantum cohomology ring of the Grassmannian. 
\end{thm}
The Pl\U cker coordinate mirror is a compactification of the EHX mirror: that is,
\begin{prop}[\cite{MarshRietsch}]\label{pro:isogr}
The Pl\U cker coordinate mirror in the rectangles cluster chart is isomorphic to the EHX mirror under the map
\[z_{ij} \mapsto \frac{p_{i \times j}}{p_{(i-1)\times {(j-1)}}}.\]
\end{prop}
This proposition and theorem can be combined to show the following theorem:
 \begin{thm}[\cite{MarshRietsch}] \label{thm:thmMR}. Let $F: \CC[z_{ij}] \to \QH^*\Gr(n,r)[s_\lambda^{-1}: \lambda \in R]$  be the map given by  $z_{ij} \mapsto \frac{s_{i \times j}}{s_{(i-1)\times {j-1}}}$. Then for any $z_{ij}$, 
\[\phi\left(\frac{\partial}{\partial z_{ij}} W_{EHX}\right)=0.\]
\end{thm}

\subsection{Superpotentials of flag varieties}
We first recall the Batyrev--Ciocan-Fontanine--Kim--van Straten generalization of the EHX mirror to flag varieties \cite{flagdegenerations}. 
Fixing $\Fl(n;r_1,\dots,r_\rho)$, for each Grassmannian step $\Gr(r_{i-1},r_i)$ draw an $r_i \times (r_{i-1}-r_i)$ grid of boxes, placing them together.  For example, the ladder diagram of $\Fl(5,4,2,1)$ is
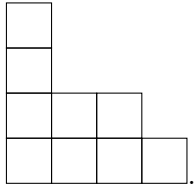
\[\begin{tikzpicture}[scale=0.6]
\draw (0,0) rectangle (1,1);
\draw (1,1) rectangle (2,2);
\draw (0,1) rectangle (1,2);
\draw (1,0) rectangle (2,1);
\draw (2,1) rectangle (3,2);
\draw (0,2) rectangle (1,3);
\draw (0,3) rectangle (1,4);
\draw (2,0) rectangle (3,1);
\draw (3,0) rectangle (4,1);
\end{tikzpicture}.\]
The dual quiver is similar to the Grassmannian  case. There are vertices inside each box, as well as at the top left and bottom right corners and in the inner corner of each step of the diagram. In this example, it is
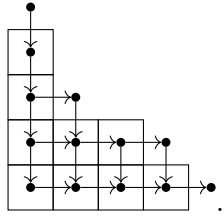
\[\begin{tikzpicture}[scale=0.6]
\draw (0,0) rectangle (1,1);
\draw (1,1) rectangle (2,2);
\draw (0,1) rectangle (1,2);
\draw (1,0) rectangle (2,1);
\draw (2,1) rectangle (3,2);
\draw (0,2) rectangle (1,3);
\draw (0,3) rectangle (1,4);
\draw (2,0) rectangle (3,1);
\draw (3,0) rectangle (4,1);
\node[circle,fill,scale=0.3] at (0.5,0.5) (a) {};
\node[circle,fill,scale=0.3] at (1.5,0.5) (b) {};
\node[circle,fill,scale=0.3] at (2.5,0.5) (c) {};
\node[circle,fill,scale=0.3] at (3.5,0.5) (d) {};
\node[circle,fill,scale=0.3] at (0.5,4.5) (l) {};
\node[circle,fill,scale=0.3] at (0.5,1.5) (e) {};
\node[circle,fill,scale=0.3] at (0.5,2.5) (f) {};
\node[circle,fill,scale=0.3] at (0.5,3.5) (g) {};
\node[circle,fill,scale=0.3] at (4.5,0.5) (h) {};
\node[circle,fill,scale=0.3] at (1.5,1.5) (i) {};
\node[circle,fill,scale=0.3] at (1.5,2.5) (j) {};
\node[circle,fill,scale=0.3] at (2.5,1.5) (k) {};
\node[circle,fill,scale=0.3] at (3.5,1.5) (m) {};
\draw[<-] (h)--(d);
\draw[<-] (d)--(c);
\draw[<-] (c)--(b);
\draw[<-] (b)--(a);
\draw[<-] (m)--(k);
\draw[<-] (k)--(i);
\draw[<-] (i)--(e);
\draw[<-] (j)--(f);

\draw[->] (e)--(a);
\draw[->] (f)--(e);
\draw[->] (g)--(f);
\draw[->] (i)--(b);
\draw[->] (j)--(i);
\draw[->] (k)--(c);
\draw[->] (l)--(g);
\draw[->] (m)--(d);
\end{tikzpicture}.\]
Assigning to each of the vertices  a variable $z_v$, the EHX superpotential is:

\[W_{EHX}=\sum_{a} \frac{z_{h(a)}}{z_{t(a)}}.\]

In \cite{kflags}, the second author proposes a generalization of the Pl\U cker coordinate mirror from Grassmannians to type A flag varieties. We recall the construction now. Fix a flag variety $\Fl(n,r_1,\dots,r_\rho)$. For each $i=1,\dots,\rho$, we can consider $r_{i-1}$  equations
 \[s^i_{\yng(1)} * s^i_\lambda=F^i_\lambda,\]
where $\lambda \in M(r_{i-1},r_i)$, and $F^i_\lambda$ is simply the expansion of the left hand side in quantum Schubert calculus. This can be described explicitly -- see \cite{kflags} for details. As in the Marsh--Rietsch construction, we can use this to obtain an expression of the anti-canonical class of the flag variety:

\[\sum_{i=1}^\rho \left(\sum_{\lambda \in M(r_{i-1},r_i)} \frac{F^i_\lambda}{s^i_{\lambda}}\right)-r_{I+1} s^i_{\yng(1)}.\]

The set $P(n,\mathbf{r})$ naturally indexes elements of the coordinate ring of the product of Grassmannians 
\[Y(n,\mathbf{r}):=\prod_{i=1}^\rho \Gr(r_{i-1},r_i).\]
Let $Q_i$ be the tautological quotient bundle pulled back to $Y(n,\mathbf{r})$ from the $i^{th}$ Grassmannian factor. Sections of $\det(Q_i)$ are indexed by $\lambda \in P(r_{i-1},r_i)$. We write $p^i_\lambda$ for the Pl\U cker coordinate associated to $i$ and $\lambda$. 

We denote $Y(n,\mathbf{r})^\circ:=\prod_{i=1}^\rho \Gr(r_{i-1},r_{i-1}-r_i)^\circ$ the locus in $Y(n,\mathbf{r})$ where $p^i_\lambda \neq 0$ for all $i$ and $\lambda \in M(r_{i-1},r_i)$. This is the complement of an anti-canonical divisor on $Y(n,\mathbf{r})$. 

To each Schubert class $s_{\vec{\mu}}$ we associate the product
\[p_{\vec{\mu}}:=\prod_{i=1}^\rho p^i_{\mu_i}.\]
We denote the polynomial in the coordinate ring of $Y(n,\mathbf{r})$ and the $q_1,\dots,q_\rho$ obtained by replacing the Schubert classes in $F^i_\lambda$ with Pl\U cker coordinates in this way as $G^i_\lambda$. 
\begin{mydef}\label{def:mirror} The Pl\U cker coordinate superpotential $W_P$ of the flag variety is
\[\sum_{i=1}^\rho \left(\sum_{\lambda \in M(r_{i-1},r_i)} \frac{G^i_\lambda}{p^i_{\lambda}}\right)-r_{I+1} p^i_{\yng(1)}.\]
\end{mydef}

\begin{eg} Consider the flag variety $\Fl(6;4,2,1)$. The Pl\U cker coordinate superpotential is
\begin{align*}
\frac{p^1_{\yng(1)}}{p^1_\emptyset}+\frac{p^1_{\yng(2,1,1,1)}}{p^1_{\yng(1,1,1,1)}}+\frac{p^1_{\yng(2,1)}}{p^1_{\yng(2)}}+\frac{p^1_{\yng(2,2,1)}+q_1 p^1_{\yng(1)}}{p^1_{\yng(2,2)}}+\frac{p^1_{\yng(2,2,2,1)}+q_1 p^1_{\yng(1,1)} p^2_{\yng(1)}}{p^1_{\yng(2,2,2)}}+\frac{q_1 p^1_{\yng(1,1,1)} p^2_{\yng(1,1)}}{p^1_{\yng(2,2,2,2)}}
\\+\frac{p^2_{\yng(1)}}{p^2_\emptyset}+\frac{p^2_{\yng(2,1)}}{p^2_{\yng(1,1)}}+\frac{p^2_{\yng(2,1)}+q_2}{p^2_{\yng(2)}}+\frac{q_2 p^2_{\yng(1)} p^3_{\yng(1)}}{p^2_{\yng(2,2)}}+\frac{p^3_{\yng(1)}}{p^3_{\emptyset}}+\frac{q_3}{p^3_{\yng(1)}}.
\end{align*}
\end{eg}
By choosing a cluster chart for each Grassmannian factor of $Y(n,\mathbf{r})$, we can expand the Pl\U cker coordinate mirror of the flag variety into algebraically independent sets of coordinates on $Y(n,\mathbf{r})$. 

In \cite{kflags}, a first check of the validity Pl\U cker coordinate mirror is carried out by demonstrating that the Pl\U cker coordinate mirror is a compactification of the EHX mirror (that is, Proposition \ref{pro:isogr} in the flag case). Fix a flag variety $\flag$. Recall that the ladder diagram is made up of the ladder diagrams of $\rho$ Grassmannians, i.e. an $r_i \times (r_{i-1} - r_i)$ grid for each $i$. Given a vertex $v$ in the $i^{th}$ block of the dual quiver, let $\phi(z_v)$ be as prescribed in the Grassmannian case for $\Gr(r_{i-1},r_i)$, and then scale by $q_1 \cdots q_{i-1}$.
\begin{eg}\label{eg:labels}
 To demonstrate, we label the vertices with $\phi(z_v)$ in the following example (where the flag variety is $\Fl(5;3,2,1)$):
\[\begin{tikzpicture}[scale=1.8]
\draw[gray] (0,0) rectangle (1,1);
\draw[gray] (1,1) rectangle (2,2);
\draw[gray] (0,1) rectangle (1,2);
\draw[gray] (1,0) rectangle (2,1);
\draw[gray] (2,1) rectangle (3,2);
\draw[gray] (0,2) rectangle (1,3);
\draw[gray] (1,2) rectangle (2,3);
\draw[gray] (2,0) rectangle (3,1);
\draw[gray] (3,0) rectangle (4,1);
\node at (0.5,0.5) (a) {$\frac{p^1_{\yng(1,1,1)}}{p^1_\emptyset}$};
\node at (1.5,0.5) (b) {$\frac{p^1_{\yng(2,2,2)}}{p^1_{\yng(1,1)}}$};
\node at (2.5,0.5) (c) {$\frac{q_1  p^2_{\yng(1,1)}}{p^2_\emptyset}$};
\node at (3.5,0.5) (d) {$ \frac{q_1 q_2 p^3_{\yng(1)}}{p^3_\emptyset}$};
\node at (0.5,1.5) (e) {$\frac{p^1_{\yng(1,1)}}{p^1_\emptyset}$};
\node at (0.5,2.5) (f) {$\frac{p^1_{\yng(1)}}{p^1_\emptyset}$};
\node at (0.5,3.5) (g) {1};
\node at (4.5,0.5) (h) {$q_1 q_2 q_3$};
\node at (1.5,1.5) (i) {$\frac{p^1_{\yng(2,2)}}{p^1_{\yng(1)}}$};
\node at (1.5,2.5) (j) {$\frac{p^1_{\yng(2)}}{p^1_\emptyset}$};
\node at (2.5,1.5) (k) {$\frac{q_1  p^2_{\yng(1)}}{p^2_\emptyset}$};
\node at (2.5,2.5) (l) {$q_1$};
\node at (3.5,1.5) (m) {$q_1 q_2$};
\draw[<-] (h)--(d);
\draw[<-] (d)--(c);
\draw[<-] (c)--(b);
\draw[<-] (b)--(a);
\draw[<-] (m)--(k);
\draw[<-] (k)--(i);
\draw[<-] (i)--(e);
\draw[<-] (l)--(j);
\draw[<-] (j)--(f);

\draw[->] (e)--(a);
\draw[->] (f)--(e);
\draw[->] (g)--(f);
\draw[->] (i)--(b);
\draw[->] (j)--(i);
\draw[->] (k)--(c);
\draw[->] (l)--(k);
\draw[->] (m)--(d);
\end{tikzpicture}.\]
\end{eg}
\begin{thm}[\cite{kflags}] \label{thm:flagMR} For any type A flag variety, the Pl\U cker coordinate mirror in the rectangles cluster chart is isomorphic to the EHX mirror under the isomorphism
\[z_v \mapsto \phi(z_v).\]
\end{thm}

\section{Quantum cohomology and mirrors of the flag variety}\label{sec:theoremA}
To summarize the situation for the Grassmannian, there are two mirrors -- the Pl\U cker coordinate mirror and the EHX mirror -- the first of which is isomorphic with the second in a particular cluster chart. Because the same partitions index Pl\U cker coordinates and Schubert classes, partial derivatives of the Pl\U cker coordinate mirror can easily be interpreted -- and indeed give -- quantum cohomology relations. 

Up until the last clause, the same is true for a multi-step flag variety: there are two mirrors -- the Pl\U cker coordinate mirror and the EHX mirror -- the first of which is isomorphic with the second in a particular cluster chart.  The same partitions index Pl\U cker coordinates and Schubert classes -- and indeed, the Abelian/non-Abelian basis of the cohomology as well. But consider the following example.

\begin{eg} The Pl\U cker coordinate mirror  of $\Fl(4;2,1)$ is
\[\frac{p^1_{\yng(1)}}{p^1_{\emptyset}}+\frac{p^1_{\yng(2,1)}+q_1}{p^1_{\yng(2)}}+\frac{p^1_{\yng(2,1)}}{p^1_{\yng(1,1)}}+\frac{q_1 p^1_{\yng(1)} p^2_{\yng(1)}}{p^1_{\yng(2,2)}}+\frac{p^2_{\yng(1)}}{p^2_{\emptyset}}+\frac{q_2}{p^2_{\yng(1)}}.\]
Expanding in the rectangles cluster and applying $p^2_{\yng(1)} \frac{\partial}{\partial p^2_{\yng(1)}}$, we obtain
\[\frac{q_1 p^1_{\yng(1)} p^2_{\yng(1)}}{p^1_{\yng(2,2)}}+p^2_{\yng(1)}-\frac{q_2}{p^2_{\yng(1)}}.\]
The most natural way to interpret this as a quantum cohomology relation is as:
\[\frac{q_1 s^1_{\yng(1)} s^2_{\yng(1)}}{s^1_{\yng(2,2)}}+s^2_{\yng(1)}-\frac{q_2}{s^2_{\yng(1)}}=0,\]
however, this relation does not hold. One could attempt to use Schubert classes instead, for example:
\[\frac{q_1 \sigma_{\yng(1),\yng(1)}}{\sigma_{\yng(2,2),\emptyset}}+\sigma_{\emptyset,\yng(1)}-\frac{q_2}{\sigma_{\emptyset,\yng(1)}}=0.\]
However, this relation also does not hold, and at any rate there will quickly be ambiguity with this approach with multi-step flag varieties.
\end{eg}
The above example demonstrates the central difficulty in the flag case: the Pl\U cker coordinate mirror is built out of quantum Schubert calculus, but is written in Pl\U cker coordinates which have the same multiplicative structure of the $s_{\vec{\mu}}$ basis. By multiplicative structure, we mean the property that the basis element associated to a tuple $(\lambda_1,\dots,\lambda_\rho)$ is the product of the $\rho$ basis elements given by tuples with a single non-empty partition $\lambda_i$ in the $i^{th}$ spot, as $i$ runs from $1$ to $\rho$.

For the flag variety, we must instead use the \emph{Schubert map}, which we introduce now. Fix a flag variety $\flag$,  where $\mathbf{r}:=r_1,\dots,r_\rho$ as usual. Recall that $P(n,\mathbf{r})$ is the set of Pl\U cker coordinates $p^i_\lambda$ on $Y(n,\mathbf{r})$, where $\lambda$ is a rectangle. Let $U_{P(n,\mathbf{r})}$ be the open subvariety of $Y(n,\mathbf{r})$ where the $p^i_\lambda, \lambda \in P(n,\mathbf{r})$ do not vanish. Let $\widetilde{QH}^*(\flag)$ denote the localization of the quantum cohomology ring at the rectangular Schubert classes.

The ring of functions $\CC[U_{P(n,\mathbf{r})}]$ is generated (as an algebra) by $P(n,\mathbf{r})$, as every Pl\U cker coordinate can be written as a Laurent polynomial in the rectangular Pl\U cker coordinates using three term Pl\U cker relations. We extend the coefficient field to the ring $R=\CC[q_1,\dots,q_\rho]$ We define a map 
\[F: R[U_{P(n,\mathbf{r})}] \to \widetilde{QH}^*(\flag)\]
-- a morphism of $\CC[q_1,\dots,q_\rho]$ algebras -- by setting the images of the rectangular Pl\U cker coordinates. 

Fix some $p^i_{j \times k}$, where the rectangle $j \times k$ is an element of $P(r_{i-1},r_i)$. We define two tuples of partitions. For $l=1,\dots,{i-1}$, let $R_l$ be the $(j-k+r_{i-1}-r_{i})\times (r_{l-1}-r_{l})$ rectangle, and set $R_i:= j \times k$. Set $\vec{\mu}_1:=(R_1,\dots,R_i,\emptyset,\dots,\emptyset)$ and  $\vec{\mu}_2:=(R_1,\dots,R_{i-1},\emptyset,\emptyset,\dots,\emptyset)$. 
\begin{mydef}\label{def:Schubertmap}The \emph{Schubert map} 
\[F: \CC[U_{P(n,\mathbf{r})}][q_1,\dots,q_\rho] \to \widetilde{QH}^*(\flag)\]
is defined by setting 
\[F(p^i_{j \times k})=\frac{\sigma_{\vec{\mu}_1}}{\sigma_{\vec{\mu}_2}}.\]
\end{mydef}
\begin{rem} Note that the Schubert map in the Grassmannian case agrees with the map defined in Theorem \ref{thm:thmMR}. 
\end{rem}
The Schubert map allows partial derivatives of the Pl\U cker coordinate mirror to be interpreted as quantum relations. We are now ready to prove Theorem \ref{thm:thmA} as stated in the introduction, which we restate here. 
\begin{theoremrepeat} $C=(C_1,\dots,C_\rho)$  be a choice of clusters for each Grassmannian factor in $Y$, and let $W_C$ be the expansion of the Pl\U cker coordinate mirror in this chart. For all $i$ and $p^i_\lambda \in C_i$, 
\[F\left(\frac{\partial}{\partial p^i_{\lambda}} W_{C}\right)=0.\]
\end{theoremrepeat}
To show this theorem will require two propositions.
\begin{prop}\label{pro:reduction} Let $C=(C_1,\dots,C_\rho)$ and $C'=(C'_1,\dots,C'_\rho)$ be two choices of clusters for $Y$ connected by a mutation. Let $W_C$ and $W_{C'}$ be the expansions of $W$ in $C$ and $C'$ respectively. Suppose Theorem \ref{thm:thmA} holds for $C$. Then it holds for $C'$.
\end{prop}
\begin{proof} For some $i=1,\dots,\rho$ there is a $\lambda, \mu, a, b, c, d \in P(r_{i-1},r_i)$ such that $C_i'$ is obtain from $C_i$ via the three term Pl\U cker relation
\[p^i_\lambda p^i_\mu = p^i_a p^i_b + p^i_c p^i_d\]
That is, $p^i_\lambda \in C_i$ and $p^i_\mu \in C_i'$, and $p^i_a,p^i_b, p^i_c$ and $p^i_d$ are elements of both $C_i$ and $C_i'$. The Laurent polynomial $W_{C'}$ is obtained from $W_C$ by replacing $p^i_\lambda$ with 
\[\frac{ p^i_a p^i_b + p^i_c p^i_d}{p^i_\mu}.\]
Note that by construction, 
\[ F(p^i_\lambda)=F\left(\frac{ p^i_a p^i_b + p^i_c p^i_d}{p^i_\mu}\right).\]

For any $ C_i'$, we can then compute 
using the multi-variable chain rule that
\[\frac{\partial}{\partial p^i_\alpha} W_{C'}=\frac{\partial}{\partial p^i_\alpha} \left(\frac{ p^i_a p^i_b + p^i_c p^i_d}{p^i_\mu}\right) \frac{\partial}{\partial p^i_\lambda} W_{C}|_{p^i_\lambda=\frac{ p^i_a p^i_b + p^i_c p^i_d}{p^i_\mu}}+\frac{\partial}{\partial p^i_\alpha} W_{C}|_{p^i_\lambda=\frac{ p^i_a p^i_b + p^i_c p^i_d}{p^i_\mu}}.\]
It follows that 
\[F\left(\frac{\partial}{\partial p^i_\alpha} W_{C'}\right)=0,\]
as 
\[F\left( \frac{\partial}{\partial p^i_\lambda} W_{C}|_{p^i_\lambda=\frac{ p^i_a p^i_b + p^i_c p^i_d}{p^i_\mu}}\right)=F\left( \frac{\partial}{\partial p^i_\lambda} W_{C}\right)=0\]
and
\[F\left( \frac{\partial}{\partial p^i_\alpha} W_{C}|_{p^i_\lambda=\frac{ p^i_a p^i_b + p^i_c p^i_d}{p^i_\mu}}\right)=F\left( \frac{\partial}{\partial p^i_\alpha} W_{C}\right)=0.\]
\end{proof}
The implication of this proposition is that we can reduce Theorem \ref{thm:thmA} to the statement for a single cluster, the rectangles cluster. The next proposition is the main ingredient in the proof of Theorem \ref{thm:thmA}, and is a corollary of the second theorem proved in this paper. This proposition uses the fact that the ladder diagram of a flag variety $\flag$ can be viewed naturally as a subquiver of the ladder diagram of a Grassmannian $\Gr(N,r_1)$, where $N>>0$ (or we can think of $\Gr(\infty,r_1)$ if we wish). 

For example, below, the ladder diagram of the flag variety $\Fl(5,3,2,1)$ is superimposed on that of $\Gr(\infty,3)$ (the second is drawn dashed in grey):
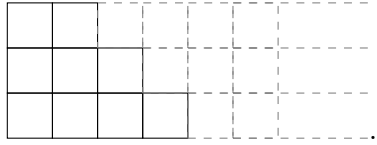
\[\begin{tikzpicture}[scale=0.6]
\draw (0,0) rectangle (1,1);
\draw (1,1) rectangle (2,2);
\draw (0,1) rectangle (1,2);
\draw (1,0) rectangle (2,1);
\draw (2,1) rectangle (3,2);
\draw (0,2) rectangle (1,3);
\draw (1,2) rectangle (2,3);
\draw (2,0) rectangle (3,1);
\draw (3,0) rectangle (4,1);

\draw[gray,dashed] (2,2) rectangle (3,3);
\draw[gray,dashed] (3,1) rectangle (4,2);
\draw[gray,dashed] (3,2) rectangle (4,3);
\draw[gray,dashed] (4,0) rectangle (5,1);
\draw[gray,dashed] (4,1) rectangle (5,2);
\draw[gray,dashed] (4,2) rectangle (5,3);
\draw[gray,dashed] (5,0) rectangle (6,1);
\draw[gray,dashed] (5,1) rectangle (6,2);
\draw[gray,dashed] (5,2) rectangle (6,3);
\draw[gray,dashed] (5,3)--(8,3);
\draw[gray,dashed] (5,2)--(8,2);
\draw[gray,dashed] (5,1)--(8,1);
\draw[gray,dashed] (5,0)--(8,0);

\end{tikzpicture}.\]

We now have two $\phi$ maps, as defined in Theorem \ref{thm:flagMR}, both with domain $\CC[z_v]$, where $v$ ranges over the vertices of the dual ladder quiver of the flag variety. Let $\phi_{\Fl}:\CC[z_v] \to \CC[U_{P(n,\mathbf{r})}]$ denote the homomorphism obtained by viewing vertices as vertices in the flag quiver. If we view a vertex as a vertex of a Grassmannian quiver, then we obtain a map  $\phi_{\Gr}$ from $\CC[z_v]$ to a localization of the coordinate ring of $\Gr(\infty, r_1)$.  More precisely, this is just the ring generated by minors of the infinite matrix
\[ \begin{bmatrix}
x_{1 1} & x_{12} & x_{13} & x_{14} & \cdots \\
x_{2 1} & x_{12} & x_{13} & x_{14} & \cdots \\
\vdots & & & \vdots \\
x_{r_1 1} & x_{r_1 2} & x_{r_1 3} & x_{r_1 4} & \cdots \\

\end{bmatrix} 
,\]
 which we can index by all partitions of length at most $r$, localized at the rectangular partitions appearing in the flag quiver.  Abusing notation, we call this ring $\CC[U_{P(\infty,r_1)}].$ By taking limits, we can see that there is a well-defined map from the ring of minors of the infinite matrix above to the symmetric polynomial ring in $r_1$ variables, $\Lambda_{r_1}$, given by
 \[ p_\lambda \mapsto s_\lambda.\]
Let $\Lambda_{r_1}^\circ$ be the localization at the rectangular coordinates. The map above gives rise to a natural generalization of the Schubert map 
\[F_{Gr}: \CC[U_{P(\infty,r_1)}] \to \Lambda_{r_1}^\circ.\]

We also have the Schubert map for the flag variety:
 \[F_{\Fl}: \CC[U_{P(n,\mathbf{r})}][q_1,\dots,q_\rho] \to \widetilde{QH}^*(\flag).\]

\begin{prop}\label{pro:commutes} Consider the natural map
\[\pi:\Lambda_{r_1}^\circ \to \widetilde{QH}^*(\flag), \hspace{5mm} s_\lambda \mapsto s^1_\lambda\]
  discussed in the introduction and in \eqref{eq:si}. Then the following diagram commutes. 

\[ \begin{tikzcd}
{} &  \CC[U_{P(\infty,r_1)}] \arrow{r}{F_{\Gr}} &\Lambda_{r_1}^\circ  \arrow{dd}{\pi}& {} \\
\CC[z_v]  \arrow{ru}{\phi_{\Gr}} \arrow[swap]{rd}{\phi_{\Fl}} & {} & {}\\
{}&\CC[U_{P(n,\mathbf{r})}][q_1,\dots,q_\rho]  \arrow{r}{F_{\Fl}}&  \widetilde{QH}^*(\flag)
\end{tikzcd}
\]

\end{prop}
\begin{eg} \label{eg:twostep-quiver}
Consider the labeled dual ladder quiver for the flag variety $\Fl(4;2,1)$:
\[\begin{tikzpicture}[scale=1.6]
\draw (0,0) rectangle (1,1);
\draw (1,1) rectangle (2,2);
\draw (0,1) rectangle (1,2);
\draw (1,0) rectangle (2,1);
\draw (2,0) rectangle (3,1);

\draw[gray,dashed] (3,2)--(3,1);
\draw[gray,dashed] (4,2)--(4,0);

\draw[gray,dashed] (2,2)--(6,2);
\draw[gray,dashed] (3,1)--(6,1);
\draw[gray,dashed] (3,0)--(6,0);

\node at (1.5,1.5) (d) {$\frac{p^1_{\yng(2)}}{p^1_\emptyset}$};

\node at (0.5,0.5) (a) {$\frac{p^1_{\yng(1,1)}}{p^1_\emptyset}$};
\node at (1.5,0.5) (b) {$\frac{p^1_{\yng(2,2)}}{p^1_{\yng(1)}}$};
\node at (2.5,0.5) (c) {$\frac{q_1  p^2_{\yng(1)}}{p^2_\emptyset}$};
\node at (0.5,1.5) (f) {$\frac{p^1_{\yng(1)}}{p^1_\emptyset}$};
\node at (2.5,1.5) (h) {$q_1$};
\node at (3.5,0.5) (i) {$q_1 q_2$};

\node at (0.5,2.5) (g) {1};

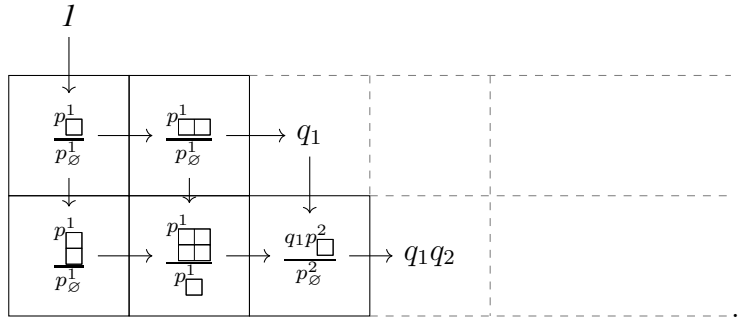
\draw[->] (g)--(f);
\draw[->] (f)--(a);
\draw[->] (f)--(d);
\draw[->] (d)--(b);
\draw[->] (a)--(b);
\draw[->] (d)--(h);
\draw[->] (b)--(c);
\draw[->] (h)--(c);
\draw[->] (c)--(i);

\end{tikzpicture}.\]

The Grassmannian labels are given by:
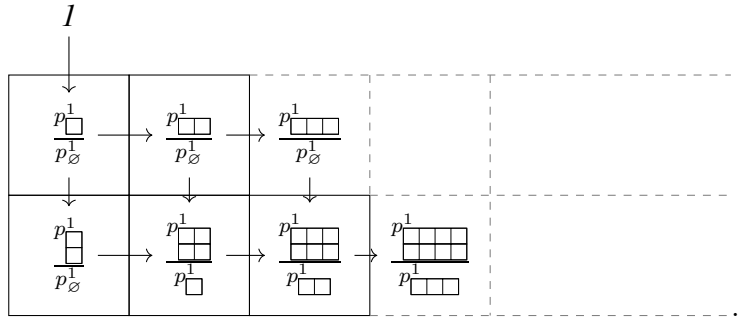
\[\begin{tikzpicture}[scale=1.6]
\draw (0,0) rectangle (1,1);
\draw (1,1) rectangle (2,2);
\draw (0,1) rectangle (1,2);
\draw (1,0) rectangle (2,1);
\draw (2,0) rectangle (3,1);

\draw[gray,dashed] (3,2)--(3,1);
\draw[gray,dashed] (4,2)--(4,0);

\draw[gray,dashed] (2,2)--(6,2);
\draw[gray,dashed] (3,1)--(6,1);
\draw[gray,dashed] (3,0)--(6,0);

\node at (1.5,1.5) (d) {$\frac{p^1_{\yng(2)}}{p^1_\emptyset}$};

\node at (0.5,0.5) (a) {$\frac{p^1_{\yng(1,1)}}{p^1_\emptyset}$};
\node at (1.5,0.5) (b) {$\frac{p^1_{\yng(2,2)}}{p^1_{\yng(1)}}$};
\node at (2.5,0.5) (c) {$\frac{p^1_{\yng(3,3)}}{p^1_{\yng(2)}}$};
\node at (0.5,1.5) (f) {$\frac{p^1_{\yng(1)}}{p^1_\emptyset}$};
\node at (2.5,1.5) (h) {$\frac{p^1_{\yng(3)}}{p^1_\emptyset}$};
\node at (3.5,0.5) (i) {$\frac{p^1_{\yng(4,4)}}{p^1_{\yng(3)}}$};

\node at (0.5,2.5) (g) {1};

\draw[->] (g)--(f);
\draw[->] (f)--(a);
\draw[->] (f)--(d);
\draw[->] (d)--(b);
\draw[->] (a)--(b);
\draw[->] (d)--(h);
\draw[->] (b)--(c);
\draw[->] (h)--(c);
\draw[->] (c)--(i);

\end{tikzpicture}.\]
Proposition \ref{pro:commutes} states that if we apply the Schubert map to the Grassmannian labels and then apply $\pi$, we obtain the same cohomology class as applying the Schubert map to the flag labels. This is trivially true for the labels in the first block.  Consider the vertex labeled $p_{\yng(3)}/p_\emptyset$. One can check using Theorem \ref{thm:thmB} (see Example \ref{eg:twostep-thm}) that
\[\pi\left(F_{\Gr}\left(\frac{p_{\yng(3)}}{p_\emptyset}\right)\right)=\frac{s^1_{\yng(3)}}{s^1_\emptyset}=q_1,\]
which is indeed the image under $F_{\Fl}$ of the label corresponding to the same vertex in the flag diagram.
Similarly, from Example \ref{eg:twostep-thm}, we also have
\[\pi\left(F_{\Gr}\left(\frac{p_{\yng(3,3)}}{p_{\yng(2)}}\right)\right)=\frac{s^1_{\yng(3,3)}}{s^1_{\yng(2)}}=\frac{q_1 \sigma_{\yng(2),\yng(1)}}{\sigma_{\yng(2),\emptyset}}=F_{\Fl}\left( \frac{q_1  p^2_{\yng(1)}}{p^2_\emptyset}\right),\]
and
\[\pi\left(F_{\Gr}\left(\frac{p_{\yng(4,4)}}{p_{\yng(3)}}\right)\right)=\frac{s^1_{\yng(4,4)}}{s^1_{\yng(3)}}=\frac{q_1^2 q_2}{q_1}=q_1 q_2=F_{\Fl}(q_1 q_2).\]

\end{eg}

To summarize, the ladder diagram of any flag variety is a subquiver of the ladder diagram of a sufficiently large Grassmannian. Using this inclusion of ladder diagrams, we can induce an inclusion of dual ladder quivers. For the Grassmannian, Theorem \ref{thm:thmMR} gives a map from vertices of the Grassmannian ladder quiver to the cohomology of the Grassmannian.  Theorem \ref{thm:flagMR} together with the Schubert map gives a map from vertices of the flag variety to the quantum cohomology of the flag variety. There is a natural map from the cohomology of the Grassmannian to the flag variety. Proposition \ref{pro:commutes} states that the Schubert map is precisely the map that makes this diagram commute. We'll delay the proof of Proposition \ref{pro:commutes} to the next section, where it will be an easy corollary of Theorem \ref{thm:thmB}.

\begin{proof}[Proof of Theorem \ref{thm:thmA}]
By Proposition \ref{pro:reduction}, it suffices to show that for $C=(C_1,\dots,C_\rho)$ the rectangles cluster, and for all $i$ and $p^i_\lambda \in C_i$, 
\[F\left(\frac{\partial}{\partial p^i_{\lambda}} W_{C}\right)=0.\]
Recall that Theorem \ref{thm:flagMR} implies that  $W_C$ can be computed using the dual ladder quiver, together with the labels as in Example \ref{eg:labels}: that is,
\begin{equation}\label{eqn:hot} W_C=\sum_{a} \frac{L(v_{t(a)})}{L(v_{s(a)})}
\end{equation}
where $a$ ranges over the arrows in the quiver, $v_{s(a)}$ and $v_{t(a)}$ are the vertices that are the source and target of the arrow $a$, and $L(v_{s(a)})$ and $L(v_{t(a)})$ the labels of these vertices. 

Fixing a rectangle $j \times k$, in either the Grassmannian or the flag case, the partial derivative $p^i_{j \times k} \frac{\partial}{\partial p^i_{j \times k}}$ can be computed using the ladder diagram as well just as in \eqref{eqn:hot}. In this case, it is a signed sum involving the arrows with source or target at one of the two vertices where $p^i_{j \times k}$ appears in the numerator or denominator of the label. That is, the sum is  over the following eight arrows, and it is a signed sum -- red arrows have a negative sign and black arrows a positive sign:
\begin{equation}\label{pic1}
\begin{tikzpicture}[scale=1.6]
\node at (0,0) (a) {$\frac{p^i_{j \times k}}{p^i_{(j-1) \times (k-1)}}$};
\node at (1,-1) (b) {$\frac{p^i_{(j+1) \times (k+1)}}{p^i_{j \times k}}$};
\node at (0,-1) (c) {};
\node at (1,0) (d) {};
\node at (0,1) (e) {};
\node at (-1,0) (f) {};
\node at (2,-1) (g) {};
\node at (1,-2) (h) {};
\draw[<-,red] (c)--(a);
\draw[<-,red] (d)--(a);
\draw[<-] (a)--(e);
\draw[<-] (a)--(f);
\draw[<-] (g)--(b);
\draw[<-] (h)--(b);
\draw[->,red] (c)--(b);
\draw[->,red] (d)--(b);
\end{tikzpicture}
\end{equation}
If a vertex is on the border of the diagram, some arrows do not appear. For example, a variable of the form $p^i_{r_i \times k}$ appears in the label of only one vertex, and that vertex is the in the bottom row. In this case, the diagram is simply
\begin{equation*}
\begin{tikzpicture}[scale=1.6]
\node at (0,0) (a) {$\frac{p^i_{j \times k}}{p^i_{(j-1) \times (k-1)}}$};
\node at (1,0) (d) {};
\node at (0,1) (e) {};
\node at (-1,0) (f) {};
\draw[<-,red] (d)--(a);
\draw[<-] (a)--(e);
\draw[<-] (a)--(f);
\end{tikzpicture}
\end{equation*}

Notice that if we consider two variables $p^i_{j \times k}$ and $p^i_{(j+1) \times (k+1)}$, the arrows involved overlap, and therefore the corresponding equations share half their terms in common. By starting with a variable of the form $p^i_{r_i \times k}$ and then consider consecutive variables 
\[p_{r_1 \times k}, p_{(r_1-1) \times (k-1)}, p_{(r_1-2) \times (k-2)}, \dots\]
for some $k$, we can easily see the partial derivatives vanish under the Schubert map if and only if diagrams of the following form vanish:
\begin{equation}\label{pic2}
\begin{tikzpicture}[scale=1.4]
\node at (0,0) (a) {$\frac{p^i_{j \times k}}{p^i_{(j-1) \times (k-1)}}$};
\node at (0,-1) (c) {};
\node at (1,0) (d) {};
\node at (0,1) (e) {};
\node at (-1,0) (f) {};
\draw[<-,red] (c)--(a);
\draw[<-,red] (d)--(a);
\draw[<-] (a)--(e);
\draw[<-] (a)--(f);
\end{tikzpicture}.
\end{equation}
Again, some arrows may not appear depending on the position of the middle vertex in the quiver. 

To summarize, it suffices to show for every internal vertex in the dual ladder quiver of the flag variety, the equation arising from \eqref{pic2} vanishes under the Schubert map. Let $E_{\Fl}$ be such an equation for a fixed vertex $v$. Let $E_{\Gr}$ be the corresponding equation for the Grassmannian for the same vertex.  By Theorem \ref{thm:thmMR}, 
\[F_{\Gr}(E_{\Gr})=0.\]Our claim is that 
\[0=\pi(F_{\Gr}(E_{\Gr}))=F_{\Fl}(E_{\Fl}). \]
If the arrows with source or target at $v$ as a vertex in the Grassmannian quiver are also arrows in the ladder quiver, this follows immediately from Proposition \ref{pro:commutes}.  For some vertices along the border, however, there may be extra arrows in the Grassmannian quiver that contribute extra terms to $E_{\Gr}$. For example, the vertex in the red box is such a vertex in the following diagram:
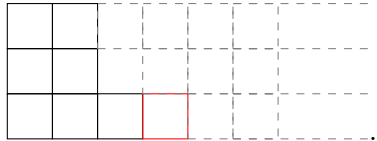
\[\begin{tikzpicture}[scale=0.6]
\draw (0,0) rectangle (1,1);
\draw (1,1) rectangle (2,2);
\draw (0,1) rectangle (1,2);
\draw (1,0) rectangle (2,1);
\draw (0,2) rectangle (1,3);
\draw (1,2) rectangle (2,3);
\draw (2,0) rectangle (3,1);
\draw[red] (3,0) rectangle (4,1);

\draw[gray,dashed] (2,2) rectangle (3,3);
\draw[gray,dashed] (3,1) rectangle (4,2);
\draw[gray,dashed] (3,2) rectangle (4,3);
\draw[gray,dashed] (4,0) rectangle (5,1);
\draw[gray,dashed] (4,1) rectangle (5,2);
\draw[gray,dashed] (4,2) rectangle (5,3);
\draw[gray,dashed] (5,0) rectangle (6,1);
\draw[gray,dashed] (5,1) rectangle (6,2);
\draw[gray,dashed] (5,2) rectangle (6,3);
\draw[gray,dashed] (5,3)--(8,3);
\draw[gray,dashed] (5,2)--(8,2);
\draw[gray,dashed] (5,1)--(8,1);
\draw[gray,dashed] (5,0)--(8,0);

\end{tikzpicture}.\]
We claim, however, that these extra terms vanish under the Schubert map, and so the above equation still holds. In the example above, the extra term in $E_{Gr}$ comes from the vertical arrow into the red box, and is
\[ \frac{p_{\yng(4,4,4)} p_{\yng(3)}}{p_{\yng(3,3)} p_{\yng(4,4)}}.\]
Note that $\pi(F_{\Gr}(p_{\yng(3)}))=s^1_{\yng(3)}=0,$ so the whole term vanishes as required. 

In general, these extra arrows come in two forms: vertical arrows along the top of a step in the ladder diagram and horizontal arrows along the side.  Fixing a block or step $i \geq 1$ of the quiver, vertical arrows contribute the factor below to an extra term:
\[\frac{p_{(r_1-r_{i+1}-1) \times (n-r_{i}+k) }}{p_{(r_1-r_{i+1}) \times (n-r_{i}+1+k) }}\]
for $k=1, \dots, r_{i}-r_{i+1}-1.$
Horizontal arrows contribute a factor of the form
\[\frac{p_{(r_1-r_{i}+k) \times (n-r_{i}+1)}}{p_{(r_1-r_{i}+k-1) \times (n-r_{i})}},\]
for $k=1,\dots, r_{i}-r_{i+1}-1.$
Since 
\[s^1_{(r_1-r_{i+1}-1) \times (n-r_{i}+k) }=0, \hspace{1mm} k=1, \dots, r_{i}-r_{i+1}-1 \]
and \[s^1_{(r_1-r_{i}+k) \times (n-r_{i}+1)}=0, \hspace{1mm} k=1,\dots, r_{i}-r_{i+1}-1\]
in $\flag$ by part (b) of Theorem \ref{thm:thmB}, the extra terms vanish as claimed. 
\end {proof}

\section{Quantum hooks and quantum cohomology}\label{sec:theoremB}

In this section, we study a natural ring homomorphism from  the ring $\Lambda_{r_1}$ of symmetric polynomials in $r_1$ variables to $\QH^*\flag$ given by mapping the $k$th  elementary symmetry polynomial in $r_1$ variables $e_k(r_1)$ to the $k$th \emph{quantum elementary polynomial} $e^q_k(r_1)$,   defined by the recursion  \eqref{eq:eq-recursion} as in Section \ref{sec:background}.

We have
a $\ZZ$-basis of $\Lambda_{r_1}$ given by Schur polynomials indexed by partitions $\lambda$ of height at most $r_1$. Using the identity $s_\lambda = \det(s_{1^{\lambda'_i+j-i}})=\det(e_{\lambda'_i+j-i}(r_1))$, where $\lambda'$ is the transpose of $\lambda$, we write $s^1_\lambda$ for the image of $s_\lambda$  under the map $\Lambda_{r_1}  \to \QH^*\flag$: 
\begin{align*}
s_\lambda & \mapsto s^1_\lambda:= \det(e^q_{\lambda'_i+j-i}(r_1)).
\end{align*}
 For $\lambda\in P(n,r_1)$, $s^1_\lambda$ represents a quantum Schubert class. When $\lambda$ has width greater than $n-r_1$, $s^1_\lambda$ is still defined, and Theorem B states that  for a particular class of partitions $\lambda$, $s^1_\lambda$ is equal to a Schubert class, up to power of $q$, and that for another class of partitions, $s^1_\lambda=0$.

We begin with some terminology. For $0 < b\leq n-r_\rho$, write  $\bar{b}:=b-(n-r_I)$ for $I$ such that $n-r_I < b\leq n-r_{I+1}$. As in the introduction, set the \emph{quantum hook} (or \emph{$q$-hook})  \emph{of width $b$} to be the partition
\[
H_b := (b^{b-n+r_1},(b-n+r_I)^{n-r_{I+1}-b}) = (b^{r_1-r_I+\bar{b}},\bar{b}^{n-r_{I+1}-b}).
\]
 In the proof of our results, we will often consider the column heights of $H_b$, which we can read from the transpose of $H_b$:
\begin{equation}\label{eq:Hb-transpose}
H'_b= ((r_1-r_{I+1})^{\bar{b}},(r_1-r_I+\bar{b})^{n-r_I}) = ((r_1-r_{I+1})^{b-n+r_I},(b-n+r_I)^{n-r_I})
\end{equation}
The $q$-hook $H_b$ can also be described as the partition obtained from a $(r_1-r_I)\times (n-r_I)$ rectangle after adding $\bar{b}$ rim-hooks of length $n+r_1-r_I-r_{I+1}$, each beginning in row $r_1-r_{I+1}$ and ending in row $1$ (see Figure \ref{fig:justqhook}, also Figure \ref{fig:qhook}).

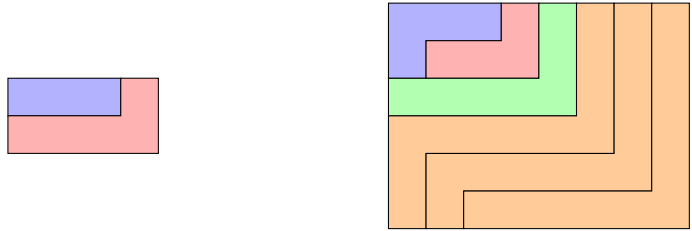
\begin{figure}[h!]
\centering
\begin{subfigure}{0.3\textwidth}
\begin{tikzpicture}[scale=.5]
\filldraw[fill=blue!30, draw=black]  (0,0) -- (3,0) -- (3,-1) -- (0,-1)-- cycle;
\filldraw[fill=red!30, draw=black] (3,0) -- (3,-1) -- (0,-1) -- (0,-2) -- (4,-2) -- (4,0) -- cycle;

\end{tikzpicture}
\end{subfigure}
\begin{subfigure}{0.3\textwidth}
\begin{tikzpicture}[scale=.5]
\filldraw[fill=blue!30, draw=black]  (0,0) -- (3,0) -- (3,-1) -- (1,-1) -- (1,-2) -- (0,-2) -- cycle;
\filldraw[fill=red!30, draw=black] (3,0) -- (3,-1) -- (1,-1) -- (1,-2) -- (4,-2) -- (4,0) -- cycle;
\filldraw[fill=green!30, draw=black] (4,0)-- (5,0) -- (5,-3) --(0,-3) --(0,-2) -- (4,-2)--cycle;
\filldraw[fill=orange!40, draw=black] (5,0) -- (6,0) -- (6,-4) -- (1,-4) --(1,-6) --(0,-6) -- (0,-3) -- (5,-3) --cycle;
\filldraw[fill=orange!40, draw=black]  (6,0)-- (7,0) -- (7,-5) -- (2,-5) -- (2,-6) -- (1,-6) -- (1,-4)--(6,-4) -- cycle;
\filldraw[fill=orange!40, draw=black] (7,0) -- (7,-5) -- (2,-5) -- (2,-6) -- (8,-6) -- (8,0) --cycle;

\end{tikzpicture}
\end{subfigure}

\caption{$q$-hooks $H_b$  of width $b=3,4$ for $\Fl(4;2,1)$ and of width $3\leq b\leq8$ for $\Fl(8;6,4,3)$. 
\label{fig:justqhook}
}

\end{figure}
For a $q$-hook $H_b$ of width $b$, set $q^{H_{b}}:=  q_1^{r_1-r_2}\cdots(q_1\cdots q_{I-1})^{r_{I-1}-r_I}(q_1\cdots q_I)^{b-(n-r_I)}.$  With this definition, note that
\begin{equation}\label{eq:qpower}
q^{H_b}=q^{H_{b-1}}\cdot q_1\cdots q_I.
\end{equation} 

\begin{eg} \label{eg:twostep}
Consider $\QH^*\Fl(4;2,1)$ with $\deg q_1=3$ and $\deg q_2=2$. For the $q$-hooks $H_3=(3,0)$ and $H_4=(4,4)$  shown in Figure \ref{fig:justqhook}, we have $q^{H_3}=q_1$ and $q^{H_4}=q_1(q_1q_2)=q_1^2q_2$.
\end{eg}
\begin{eg}  \label{eg:qhook}
Consider   $\QH^*\Fl(8;6,4,3)$ with $\deg q_1=4, \deg q_2=3$, and $\deg q_3=4$. Let $I$ be such that $n-r_I < b\leq n-r_{I+1}$. For the $q$-hooks of width $3\leq b\leq8$ (depicted in Figure \ref{fig:justqhook}),
we have:
\begin{table}[h]
\[
\begin{array}{|c|c|c|c|} \hline 
b & I & H_b & q^{H_b}\\ \hline \hline
3& 1& (3,1)   & q_1\\ \hline
4 & 1& (4,4)  & q_1^2 \\ \hline
5&2& (5,5,5)   & q_1^2(q_1q_2)=q_1^3q_2\\ \hline
6 & 3& (6,6,6,6,1,1) & q_1^2(q_1q_2)(q_1q_2q_3)= q_1^4q_2^2q_3 \\ \hline
7 & 3& (7,7,7,7,7,2) & q_1^2(q_1q_2)(q_1q_2q_3)^2= q_1^5q_2^3q_3^2 \\ \hline
8 & 3& (8,8,8,8,8,8) &q_1^2(q_1q_2)(q_1q_2q_3)^3=  q_1^6q_2^4q_3^3\\ \hline
\end{array}
\]
\end{table}
\end{eg}
Let $R_b := (b^{r_1-r_I+\bar{b}}) = (b^{r_1-(n-b)})$ be the maximal rectangle of width $b$ contained in $H_b$,  with  $H_b=R_b=\emptyset$ if $b<n-r_1$. 
\begin{rem}\label{rem:qhook}
For a partition  $\lambda\subseteq r_1\times n$,  let $I$ be such that $n-r_I < \lambda_1\leq n-r_{I+1}$. Then $R_{\lambda_1}\subseteq \lambda$ if condition (i) below holds, and $H_{\lambda_1}\subseteq \lambda$ if conditions (i) and (ii) below hold.
\begin{enumerate}
\item[(i)] $\lambda_{\lambda_1}'\geq \lambda_1-(n-r_1)$ 
\item[(ii)]$ \lambda_{\lambda_1-(n-r_I)}' \geq r_1 -r_{I+1}$.
\end{enumerate}
    (Note that if $\lambda_1=n-r_{I+1}$, then condition (ii) is redundant.)
  \end{rem}
Conditions (i) and (ii) are illustrated  in the left  diagram of Figure \ref{fig:qhook} by $\lambda$ containing the southeast corner boxes of the $q$-hook marked by $+$ and $\times$, respectively. Here, $b=\lambda_1$, $\bar{b}:=\lambda_1-(n-r_I)$, and $r_1-r_I+\bar{b}=\lambda_1-(n-r_1)$.

\begin{mydef} \label{def:compatible}A  partition  $\lambda\subseteq r_1\times n$ is \emph{compatible with a $q$-hook} if $H_{\lambda_1}\subseteq \lambda$, i.e. conditions (i) and (ii) of Remark \ref{rem:qhook} holds.
\end{mydef}

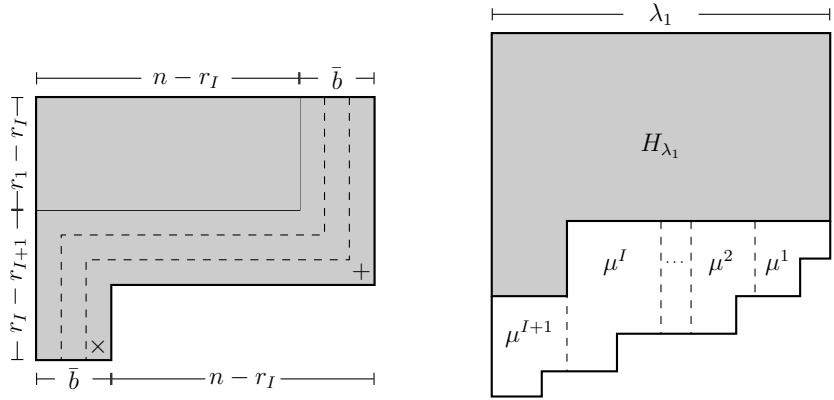
\begin{figure}[h!]

\centering
\begin{tikzpicture}[scale=.5]
\draw  (0, -3) -- (7, -3) -- (7, 0) --  (9, 0) -- (9,-5) -- (2, -5) -- (2, -7)-- (0, -7) -- cycle;
\fill[color=gray!40]    (0, -3) -- (7, -3) -- (7, 0) -- (9, 0) -- (9,-5) -- (2, -5) -- (2, -7)-- (0, -7) -- cycle;

\draw[thick] (0, 0) rectangle (7, -3);
\fill[color=gray!40] (0, 0) rectangle (7, -3);
\draw[thick] (6, 0) -- (9, 0) -- (9,-5) -- (2, -5) -- (2, -7) -- (0, -7)--(0,-3) -- (0,0) -- (6,0);
\draw[dashed] (7.67,0) -- (7.67,-3.67) -- (.67,-3.67) -- (.67,-7);
\draw[dashed] (8.33,0) -- (8.33,-4.33) -- (1.33,-4.33) -- (1.33,-7);

\node[scale=.8] at (1.65, -6.65) {$\times$};
\node[scale=.8] at (8.65, -4.65) {$+$};

\node[scale=.8] at (4, .5) {$n-r_I$};
\draw [-|] (5.2, .5) -- (7, .5);
\draw [|-] (0, .5) -- (2.8, .5);

\node[scale=.8] at (8, .5) {$\bar{b}$};
\draw [-|] (8.5, .5) -- (9, .5);
\draw [|-] (7, .5) -- (7.5, .5);

\node[scale=.8, rotate = 90] at (-.5, -1.5) {$r_1-r_I$};
\draw [-|] (-.5, -2.5) -- (-.5, -3);
\draw [|-] (-.5, 0) -- (-.5, -.5);

\node[scale=.8, rotate = 90] at (-.5, -5) {$r_I-r_{I+1}$};
\draw [-|] (-.5, -6.5) -- (-.5, -7);
\draw [|-] (-.5, -3) -- (-.5, -3.5);

\node[scale=.8] at (1, -7.5) {$\bar{b}$};
\draw [-|] (1.5, -7.5) -- (2, -7.5);
\draw [|-] (0, -7.5) -- (.5, -7.5);

\node[scale=.8] at (5.5, -7.5) {$n-r_I$};
\draw [-|] (6.5, -7.5) -- (9, -7.5);
\draw [|-] (2, -7.5) -- (4.5, -7.5);

\end{tikzpicture}
\hspace{.5in}
\begin{tikzpicture}[scale=.5]
\fill[color=gray!40]   (0,0) --(9, 0) -- (9,-5) -- (2, -5) -- (2, -7)-- (0, -7) -- cycle;

\draw[thick] (6, 0) -- (9, 0) -- (9,-5) -- (2, -5) -- (2, -7) -- (0, -7)--(0,-3) -- (0,0) -- (6,0);

\draw[thick] (0,-7) --  (0,-9.67) -- (1.33,-9.67)-- (1.33,-9) -- (2,-9)-- (3.33,-9) -- (3.33,-8)  -- (6.5,-8)--(6.5,-7)--(8.2,-7) --  (8.2,-6)--(9,-6) --(9,-5);

\draw[dashed]  (2,-6) -- (2,-9);
\draw[dashed]  (4.5,-5) -- (4.5,-8);
\node[scale=.5] at (4.9, -6.1) {$\cdots$};
\draw[dashed]  (5.3,-5) -- (5.3,-8);
\draw[dashed] (7,-5) -- (7,-7);
\node[scale=.8]  at (1, -8) {$\mu^{I+1}$};

\node[scale=.8]  at (3.3,-6.1) {$\mu^I$};

\node[scale=.8] at (6.1,-6.1) {$\mu^2$};
\node[scale=.8] at (7.6,-6.1) {$\mu^1$};

\node[scale=.8] at (4.5,-3) {$H_{\lambda_1}$};

\node[scale=.8] at (4.5, .5) {$\lambda_1$};
\draw [-|] (5.5, .5) -- (9, .5);
\draw [|-] (0, .5) -- (3.5, .5);

\end{tikzpicture}
\caption{A $q$-hook $H_b$  of width $b$ and a skew shape $\lambda/H_{\lambda_1}$ with associated tuple of partitions $\vec{\mu}_\lambda=(\mu^1,\ldots,\mu^{I+1},\emptyset,\ldots,\emptyset)$.  }
\label{fig:qhook}
\end{figure}

\begin{rem}\label{rem:emptyhook} The partition $H_b$ has height $r_1-r_{I+1}$. By convention, $r_0=n$, so when $0=n-r_0< b\leq n-r_1$, $H_b$ is the empty partition, and so every partition $\lambda$ of width at most $n-r_1$ is compatible with a $q$-hook.
\end{rem}

 For  a partition $\lambda\subseteq r_1\times n$ that is compatible with a $q$-hook,  define partitions $\mu^1,\ldots,\mu^I$   by subdividing the skew shape $\lambda/H_{\lambda_1}$, where $\mu^1$ is the partition consisting of the rightmost $n-r_1$ columns of $H_{\lambda_1}$, $\mu^2$ is the partition consisting of the second rightmost $r_1-r_2$ columns of $H_{\lambda_1}$, etc. If $I<\rho$, let  $\mu^{I+1}$ be the partition consisting of the leftmost $\bar{b}$ columns.  (See Figure \ref{fig:qhook}.)

\begin{mydef} \label{def:mu-lambda} 
For  a partition $\lambda\subseteq r_1\times n$ that is compatible with a $q$-hook,  define \emph{the tuple of partitions associated to $\lambda/H_{\lambda_1}$} to be $\vec{\mu}_\lambda= (\mu^1,\ldots,\mu^{I+1},\emptyset,\ldots,\emptyset)$ if $I<\rho$ and  $\vec{\mu}_\lambda= (\mu^1,\ldots,\mu^{I})$ if $I=\rho$, as described above (see Figures \ref{fig:qhook} and \ref{fig:eg}).  Here, $\vec{\mu}_\lambda\in P(n,\mathbf{r})$ since $\mu^\ell\subseteq {r_\ell\times (r_{\ell-1}- r_\ell)}$ for $1\leq l\leq \rho$. \end{mydef}

\begin{lem}\label{lem:bijections} For  a partition $\lambda$ that is compatible with a $q$-hook, let $w$ be the (321-avoiding) permutation
corresponding to $(\lambda/H_{\lambda_1},\omega)$ with labeling $\omega(i,j)=r_1+i-j$ under the bijection in \cite{bjs}.  Then $w$ is equal to the permutation corresponding to the tuple $\vec{\mu}_\lambda$ via the bijection described in Remark \ref{rem:bijection}. Moreover,  $w$  is either Grassmannian with descent at $r_{I+1}$ or has descents at exactly $r_I$ and $r_{I+1}$, where  $I$ is such that $n-r_{I}<\lambda_1\leq n-r_{I+1}$.
\end{lem}
\begin{proof} A reduced expression for the (321-avoiding) permutation $w$ corresponding to $(\lambda/H_{\lambda_1},\omega)$ is given by \cite{bjs} as the product of simple transpositions obtained from reading the labeling from bottom to top, beginning with the rightmost column.
This product respects the subdivision of $\lambda/H_{\lambda_1}$ into the tuple of labeled partitions $\mu^1,\ldots,\mu^I,\mu^{I+1}$ with labeling $\omega^\ell(i,j) = r_\ell + i - j$ for $1\leq \ell\leq I+1$. (From Definition \ref{def:mu-lambda}, if $I=\rho$, then the tuple consists of only $\mu^1,\ldots,\mu^I$.)  Again, by \cite{bjs} (see also \cite{kmy}), a reduced word for $\mu^\ell$ is the product of simple transpositions obtained by reading the labeling of $\mu^\ell$ from bottom to top, beginning with the rightmost column. Concatenating these expressions recovers $w$. Moreover, define the partition $\mu^{[I]}$ to be the partition obtained by appending the partitions $\mu^I,\ldots,\mu^1$ together; this consists of the last $n-r_I$ columns of $\lambda/H_{\lambda_1}$. Let $w^{I+1}$ and $w^{[I]}:=w^1\cdots w^I$ be the Grassmannian permutations associated to $\mu^{I+1}$ and $\mu^{[I]}$; these have possible descents at $r_{I+1}$ and $r_{I}$, respectively, and so their product has possible descents at only $r_{I+1}$ and $r_{I}$.

\end{proof}
\begin{rem} \label{rem:lambda-perm} For a partition $\lambda$ that is compatible with a $q$-hook with
corresponding tuple $\vec{\mu}_\lambda$ and permutation $w$, we denote the associated Schubert class by $\sigma_{\vec{\mu}_\lambda}$, $\sigma_{w}$, or simply  $\sigma_{\lambda/H_{\lambda_1}}$. \end{rem}

\begin{eg} \label{eg:twostep-cont}
Consider $\Fl(4;2,1)$ as in Example \ref{eg:twostep}. The partition $(3,3)$ is compatible with the $q$-hook $H_3=(3,0)$. 
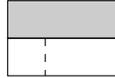
\[
\begin{tikzpicture}[scale=.5]
\filldraw[fill=gray!40, draw=black]  (0,0) -- (3,0) -- (3,-1) -- (0,-1)-- cycle;
\draw (0,-1) rectangle (3,-2);
\draw[dashed] (1,-1)--(1,-2);
\end{tikzpicture}
\]
The associated tuple of partitions $(\yng(2),\yng(1))$ is read from right to left from the skew shape $(3,3)/H_3$.
\end{eg}
\begin{eg} \label{eg:qhook-cont}Consider $\Fl(8;6,4,3)$ as in Example \ref{eg:qhook} and partitions 
$\eta=(3,3,3,2), \lambda=(5,5,5,5,5,4,2)$ and $\nu= (6,6,6,6,5,3)$. Then $\eta$ is compatible with the $q$-hook $H_3=(3,1)$, $\lambda$ is compatible with the $q$-hook $H_5= (5,5,5)$ and $\nu$ is compatible with the $q$-hook $H_6=(6,6,6,6,1,1)$. The associated tuples of partitions to $\eta/H_3$, $\lambda/H_5$ and $\nu/H_6$ are 
$\vec{\mu}_\eta=(\yng(2,1),\yng(1,1),\emptyset)$, $\vec{\mu}_\lambda=\left( \yng(2,1),\yng(2,2,1),\yng(1,1,1)\right)$ and $\vec{\mu}_\nu=\left(\yng(1),\yng(2,1),\yng(1,1)\right) $, as seen in Figure \ref{fig:eg} by reading the associated tuple of partitions from right to left. Note that $I=\rho=3$ in Definition \ref{def:mu-lambda} for $\nu$ since $n-r_3=5<\nu_1$. Also note  that as in Lemma \ref{lem:bijections} and Example \ref{eg:permutation}, the first permutation has descents at $r_1=6$ and $r_2=4$ and the other two permutations are Grassmannian with descent at $r_3=3$.
\end{eg}

\begin{figure}[h!]

\centering
\begin{tikzpicture}[scale=.5]
\filldraw[fill=gray!40, draw=black]  (0,0) -- (3,0) -- (3,-1) -- (1,-1) -- (1,-2) -- (0,-2) -- cycle;
\draw  (3,-1) --  (3,-2) -- (2,-2) -- (2,-3) -- (1,-3) -- (1,-4) --(0,-4)-- (0,-2);
\draw[dashed] (1,-2) -- (1,-4);
\node[scale=.8]  at (1.5, -.5) {$H_3$};
\end{tikzpicture}
\hspace{.5in}
\begin{tikzpicture}[scale=.5]

\filldraw[fill=gray!40, draw=black] (0,0) rectangle (5,-3);
\draw  (5,-3) -- (5,-4) -- (4,-4) -- (4,-5) --(2,-5) --(2,-6) -- (0,-6) -- (0,-3) --cycle;
\draw[dashed] (3,-3) -- (3,-5);
\draw[dashed] (1,-3) -- (1,-6);
\node[scale=.8]  at (2.5, -1.5) {$H_5$};
\end{tikzpicture}
\hspace{.5in}
\begin{tikzpicture}[scale=.5]
\filldraw[fill=gray!40, draw=black] (0,0) -- (6,0) -- (6,-4) -- (1,-4) --(1,-6) --(0,-6) -- (0,-3) --cycle;
\draw   (5,-4)-- (5,-5) -- (3,-5) -- (3,-6) -- (1,-6)--(1,-4) --cycle;
 \draw[dashed] (4,-4) -- (4,-5);
\draw[dashed] (2,-4) -- (2,-6);
\node[scale=.8]  at (3, -2) {$H_6$};

\end{tikzpicture}

\caption{Partitions $\eta , \lambda$, and $\nu$, skew shapes $\eta/H_3$, $\lambda/H_5$ and $\nu/H_6,$ and their associated tuples of partitions $\vec{\mu}_\eta, \vec{\mu}_\lambda$ and $\vec{\mu}_\nu$ for $\Fl(8;6,4,3)$.}
\label{fig:eg}

\end{figure}
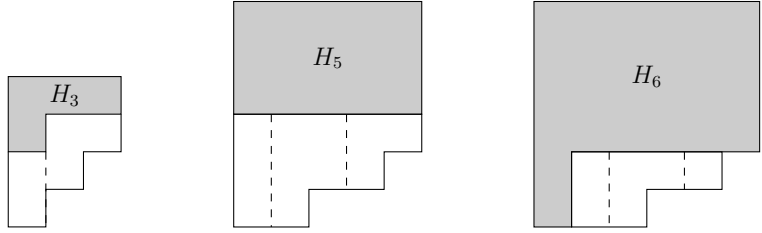

Before proving Theorem \ref{thm:thmB}, we introduce and study the following auxiliary partitions.
\begin{mydef} \label{def:qlambda}
Given a partition $\lambda\subseteq r_1\times n$ and $1\leq m\leq \lambda_1$, define $\lambda^{(m)}$ to be the partition  obtained from $\lambda$ by removing column $m$ from $\lambda$ and adding 1 to columns $1,\ldots,m-1$, i.e.
\[(\lambda^{(m)})' =(\lambda_1'+1,\cdots, \lambda_{m-1}'+1,{\lambda}_{m+1}'+1,\cdots,\lambda_{\lambda_1}'),
\]
where $\lambda'$ is the transpose of $\lambda$, i.e.
$(\lambda^{(m)})'_i  = \lambda_i'+1$ for $i<m$ and $(\lambda^{(m)})'_i=\lambda_{i+1}$ for $i\geq m$. (See Figure \ref{fig:lambdamhook}.)

\end{mydef}

\begin{figure}[h!]

\begin{tikzpicture}[scale=.6]

\node[scale=.7] at (4.5, .5) {$\lambda_1$};
\draw [-|] (5, .5) -- (9, .5);
\draw [|-] (0, .5) -- (4, .5);

\draw[thick] (9,-5) -- (2, -5) -- (2, -7)-- (0, -7) -- (0,-9.67) -- (.67,-9.67) -- (.67,-9)-- (2,-9)-- (3.33,-9) -- (3.33,-7.67) -- (6.5,-7.67)--(6.5,-7)--(8.2,-7) --  (8.2,-6)--(9,-6) -- cycle;

\fill[color=orange!40]    (0, 0)--(9,0) -- (9,-4.33) -- (1.33,-4.33) -- (1.33,-7)-- (0, -7) -- cycle;

\node[scale=.8] at (1.65, -6.65) {$\times$};
\node[scale=.8] at (8.65, -4.65) {$+$};
\node[scale=.8] at (8.65, -4) {$\oplus$};
\node[scale=.8] at (1, -6 ) {$\otimes$};

\draw[thick] (6, 0) -- (9, 0) -- (9,-5) -- (2, -5) -- (2, -7) -- (0, -7)--(0,-3) -- (0,0) -- (6,0);

\draw    (1.33,-7)  --  (1.33,-4.33)  -- (9,-4.33 ) ;

\draw[pattern=north west lines] (4.67,0) rectangle (5.33,-7.67);
\node[scale=.7] at (7.7,-8.5) {{$m$th column of $\lambda$}};
\draw[->] (5.7,-8.5) to [out=20,in=10, bend left, out looseness=1]  (5 ,-7.8);

\draw[fill=red!30] (0,-9.67) rectangle  (.67,-10.33);
\draw[fill=red!30] (.67,-9) rectangle  (3.33,-9.67);
\draw[fill=red!30] (3.33,-7.67) rectangle  (4.67,-8.33);

\end{tikzpicture}
\caption{The skew shape $\lambda/H_{\lambda_1}$ and the skew shape $\lambda^{(m)}/H_{\lambda_1-1}$.  }
\label{fig:lambdamhook}
\end{figure}
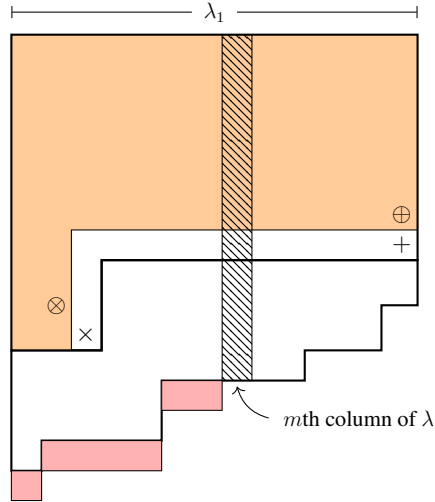

\begin{lem}\label{lem:lambdamcompatible}
If a partition $\lambda\subseteq r_1\times n$ is compatible with a $q$-hook, then $\lambda^{(m)}$ is compatible with a $q$-hook for $1\leq m\leq \lambda_1$.
\end{lem}
\begin{proof}
This follows from Remark \ref{rem:qhook} and Definition \ref{def:qlambda}, where conditions (i) and (ii) of Remark \ref{rem:qhook} for $\lambda^{(m)}$ are illustrated by $\otimes$ and $\oplus$ in Figure \ref{fig:lambdamhook}.
\end{proof}

For a partition $\lambda\subseteq r_1\times n$, consider $s^1_\lambda :=\det (s^1_{\lambda'_i+j-i})$ and the determinants $\Delta$ as in  \eqref{eq:q-determinant} and \eqref{eq:si}.

\begin{prop}
\label{prop:expand}
Given a partition $\lambda\subseteq r_1\times n$   of width $b:=\lambda_1$,   let $I$ be such that $n-r_{I}<b\leq n-r_{I+1}$ and let $\bar{b}=b-(n-r_I)$. Then 
\[
\sum_{m=1}^b (-1)^{m-1} s^1_{1^{\lambda'_m-m+1}}  * \Delta_{\lambda^{(m)}/H_{b-1}}(\overline{\psi}) 
 = q_1\cdots q_I \, \Delta_{\lambda/H_b}(e^q(\phi)) \text{ in } \QH^*\flag,\]
where $\phi=((I+1)^{\bar{b}},I^{r_{I-1}-r_I},\ldots,1^{n-r_1})$ and
$\overline{\psi}=((I+1)^{\bar{b}-1}, I^{r_{I-1}-r_I},\ldots,1^{n-r_1})$. 
\end{prop}

\begin{proof}
From  \eqref{eq:q-determinant}, we have    
\begin{align}
 \Delta_{\lambda/H_b}(\phi) &= \det\left(e^q_{\lambda'_i-(H_b)'_j+j-i}(\phi_j) \right)=: \det[v_1,\ldots,v_b]
 \label{eq:determinant}  \\
 \Delta_{\lambda^{(m)}/H_{b-1}}(\overline{\psi}) &=\det\left(e^q_{(\lambda^{(m)})'_i-(H_{b-1})'_j+j-i}(\overline{\psi}_j) \right),  
  \label{eq:mdeterminant}
  \end{align}
where $\phi$ and $\overline{\psi}$ are as in the statement of the proposition, and where we write $v_j$ for the $j$th column of the matrix in \eqref{eq:determinant}.
(Note that $\lambda$  and $\lambda^{(m)}$ need not contain $H_b$ and $H_{b-1}$, respectively.)
 Since $s^1_{1^a} = e^q_a(r_1)$ in $\QH^*\flag$, the left hand quantity of the proposition can be rewritten as 
the determinant:
\begin{equation}\label{eq:intermediate}
\det\left(e^q_{\lambda'_i-(0,(H_{b-1})')_j+j-i}( 1,\overline{\psi}_j )\right)_{i,j}
\end{equation}
where 
$(1,\overline{\psi})=(1,(I+1)^{\bar{b}-1}, I^{r_{I-1}-r_I},\ldots,1^{n-r_1})$.
We proceed by reordering the columns of this matrix and then comparing the resulting determinant to \eqref{eq:determinant}. More concretely, let $\tau\in S_b$  be the permutation defined by 
\[ \tau(j) =\left\{
\begin{array}{ll}
{\bar{b}+r_1-r_I} &\text{if } j=1 \\
j & \text{if } 2\leq j\leq \bar{b} \\
 {\bar{b}+r_{\ell+1}-r_I +1}&\text{if } j = \bar{b}+r_{\ell-1}-r_I \text{ for } 1\leq l< I\\
 1 &\text{if }  j = \bar{b}+r_{I-1}-r_I\\
{j+1} &\text{otherwise}. 
\end{array}
\right.
\]
Reordering  columns  using the permutation $\tau$, from the description of $H'_b$ and $H'_{b-1}$ in \eqref{eq:Hb-transpose},  \eqref{eq:intermediate}  is equal to ${\sgn(\tau})$ times
\begin{equation} \label{eq:otherdeterminant}
 \det\left(e^q_{\lambda'_i-\kappa_j+j-i}( \psi_j )\right)_{1\leq i,j\leq b}  =: \det[\overline{v}_1,\ldots,\overline{v}_b] ,
\end{equation}
where $\displaystyle \kappa=H'_b - (r_I-r_{I+1}){\bf e}_1-\sum_{1\leq l<I}(r_{\ell-1}-r_\ell){\bf e}_{\bar{b}+1+r_\ell-r_I}$
 and $\displaystyle\psi=\phi- {\bf e}_1 -\sum_{1\leq l<I}{\bf e}_{\bar{b}+1+r_\ell-r_I}$.
 Here, ${\bf e}_j$ denotes the sequence that is $1$ in position $j$ and $0$ elsewhere, and $\overline{v}_j$ is the $j$th column of the determinant in \eqref{eq:otherdeterminant}.
Thus, column $v_j$  of  \eqref{eq:determinant} is equal to  column $\bar{v}_j$ of \eqref{eq:otherdeterminant} except when $j=1$ or  $j=\bar{b}+1+(r_{\ell-1}-r_I)$ with $1\leq l<I$.

We rewrite  \eqref{eq:eq-recursion}  as
 \begin{equation}\label{eq:eq-recursion-l}
e^{q}_a(r_\ell) = e^{q}_a(r_{\ell-1}) -\left( \sum_{m=1}^{r_{\ell-1}-r_\ell} \sigma_m^l e^{q}_{a-m}(r_\ell) \right) + (-1)^{r_{\ell-1}-r_\ell}q_\ell \, e^{q}_{a-(r_{\ell-1}-r_{\ell+1})}(r_{\ell+1}).
\end{equation}
Note that for $l=1$,  the first term vanishes since $e^q_a(r_0)=e^q_a(n)=0$ is a relation in the quantum cohomology ring for all $a$.

We now describe the transition matrix between vectors $v_j$ and $v'_j$.
Consider the $b\times b$ matrix $A=(a_{ij})$ with entries 
\begin{equation}
\label{eq:A-matrix-defn}
a_{ij} = \begin{cases} 
-\sigma_{\bar{b}+r_{I-1}-r_I+1-i}^I&\text{ if } j=1\\
-\sigma^l_{\bar{b}+r_{\ell-1}-r_I+1-i}&\text{ if } j=\bar{b}+1+r_{\ell+1}-r_I  \text{ for } 1\leq l<I
  \\  0 & \text{ otherwise},
  \end{cases}
\end{equation}
with the convention that $\sigma_0^l=1$ and $\sigma_i^l=0$ for $i<0$ and $i>{r_{\ell-1}-r_\ell}$. Then $A$ is a lower triangular matrix with zeros along the diagonal. Let $D = (d_{ij})$ be the $b\times b$ diagonal matrix with entries 
\begin{equation}
\label{eq:D-matrix-defn}
 d_{jj} = \begin{cases} 
  (-1)^{r_{I-1}-r_{I}}q_I &\text{ if } j=1\\
 (-1)^{r_{\ell-1}-r_\ell}q_\ell &\text{ if } j=\bar{b}+1+r_{\ell+1}-r_I  \text{ for } 1\leq l<I
  \\  1  & \text{ otherwise}.
  \end{cases}
\end{equation}
With this notation, the relation between the vectors $v_j$ and $v'_j$ is given by matrix multiplication
\[ 
[\overline{v}_1,\ldots,\overline{v}_b] =(A+D) [v_1,\ldots,v_b].\]
Since $A+D$ is lower triangular with diagonal entries $d_{jj}$, $\det(A+D) = \prod_{\ell=1}^I (-1)^{r_{\ell-1}-r_\ell}q_\ell$, and so
\[ 
\det[\overline{v}_1,\ldots,\overline{v}_b] =  (-1)^{n-r_1}q_1\ldots q_I \cdot \det[v_1,\ldots,v_b] \]
and hence by \eqref{eq:intermediate} and \eqref{eq:otherdeterminant},
the left hand side of the proposition is equal to
\[  (-1)^{n-r_1+I}q_1\ldots q_I \,\sgn(\tau) \det[v_1,\ldots,v_b]. \]
Since the signature $\sgn(\tau)$of the permutation $\tau$ is $(-1)^{n-r_I}$, we conclude that \eqref{eq:otherdeterminant} is equal to $q_1\cdots q_I$ times the determinant \eqref{eq:determinant}, as needed.

\end{proof}

\section{Proof of Theorem \ref{thm:thmB}}

In this section, we use the set up and results from Section \ref{sec:theoremB}, including Proposition \ref{prop:expand}, to prove Theorem \ref{thm:thmB}, which we restate here.

\begin{theoremrepeat}
Let $\lambda\subseteq r_1\times n$ be a partition, and let $I$ be such that $n-r_I<\lambda_1\leq n-r_{I+1}$. 
\begin{enumerate}
 \item[(a)]  If $H_{\lambda_1}\subseteq \lambda$, then
\[ s^1_\lambda = q^{H_{\lambda_1}}\sigma_{\vec{\mu}} \text{ in } \QH^*\flag,
\]
 where $\vec{\mu}=(\mu^1,\ldots,\mu^{I+1},\emptyset,\ldots,\emptyset)$ is the tuple of partitions associated to $\lambda$ above.  In particular,  $s^1_{H_{b}} = q^{H_{b}}$ since  $H_{b}/H_{b}=\emptyset$, so  $\mu^j=\emptyset$ for all $j$ and $\sigma_{(\emptyset,\ldots,\emptyset)} = 1$.

 \item[(b)] If $\lambda$ contains $R_{\lambda_1}$, but $H_{\lambda_1}\not\subseteq \lambda$, then  
$$s^1_\lambda =0 \text{ in } \QH^*\flag.$$
 \end{enumerate}
 \end{theoremrepeat}

\begin{eg}\label{eg:twostep-thm}
For $\QH^*\Fl(4;2,1)$, by part (a) of the theorem and Examples \ref{eg:twostep} and \ref{eg:twostep-cont}, we have 
\begin{align*}
s^1_{\yng(2)} & = \sigma_{\yng(2),\emptyset} \\
s^1_{\yng(3)} & = q_1 \\ 
s^1_{\yng(3,3)} &= q_1\sigma_{\yng(2),\yng(1)}\\
s^1_{\yng(4,4)} & = q_1^2q_2. 
\end{align*}
(See also Example \ref{eg:twostep-quiver}.)
\end{eg}
\begin{eg} \label{eg:thm}
Consider $\QH^*\Fl(8;6,3,2)$ as in Examples \ref{eg:qhook} and \ref{eg:qhook-cont}. By part (a) of the theorem, for the partitions 
$\eta=(3,3,3,2),\lambda=(5,5,5,5,5,4,2)$ and $\nu= (6,6,6,6,5,3)$, we have
\[ s^1_\eta=q_1\sigma_{\yng(2,1),\yng(1,1),\emptyset}, \,
s^1_\lambda = q_1^3q_2 \sigma_{\yng(2,1),\yng(2,2,1),\yng(1,1,1)} \, \text{ and } \, s^1_\nu = q_1^4q_2^2q_3 \sigma_{\yng(1),\yng(2,1),\yng(1,1)}.
\]
From Remark \ref{rem:bijection} and Example \ref{eg:permutation}, we can also write this in terms of the indexing of Schubert classes by permutations  as 
\[ s^1_\eta=q_1\sigma_{12453768}, \,s^1_\lambda = q_1^3q_2 \sigma_{36812457}\, \text{ and } \, s^1_\nu = q_1^4q_2^2q_3 \sigma_{14723568 } .
\]
On the other hand, for the partition $\gamma=(6,6,6,6,3)$, we have 
$s^1_\gamma=0$ by part (b) of the theorem since $\gamma$ contains $R_6=(6,6,6,6)$ but not $H_6=(6,6,6,6,1,1)$.
\end{eg}

We now prove part (a) of Theorem \ref{thm:thmB} and then use part (a) to prove part (b).

\begin{proof}[Proof of part (a) of Theorem \ref{thm:thmB}]
We proceed by induction on the width $b:=\lambda_1$ of $\lambda$.  For the base cases, when $0<b\leq n-r_1$, by Remark \ref{rem:emptyhook}, $H_b$ is the empty partition, and we have the equality $s^1_\lambda=\sigma_{\lambda}=\sigma_{\lambda/\emptyset}$.

Now assume the result for partitions of width at most $b-1$. Given a partition $\lambda\subseteq r_1\times n$,  expanding the determinant $s^1_\lambda :=\det (s^1_{\lambda'_i+j-i})$ along the first column gives
\begin{equation} \label{eq:expansion0}
s^1_\lambda  =\sum_{m=1}^b (-1)^{m-1} s^1_{1^{\lambda'_m-m+1}} * s^1_{\lambda^{(m)}}.
\end{equation}
From Lemma \ref{lem:lambdamcompatible}, $\lambda^{(m)}$ is compatible with a $q$-hook, so
 by the induction hypothesis,
$s^1_{\lambda^{(m)} }= q^{H_{b-1}}\sigma_{\lambda^{(m)}/H_{b-1}}$, and
\eqref{eq:expansion0} becomes
\begin{align*}
 s^1_\lambda  &=\sum_{m=1}^b (-1)^{m-1} s^1_{1^{\lambda'_m-m+1}} *    q^{H_{b-1}}\sigma_{\lambda^{(m)}/H_{b-1}} \\
 &= q^{H_{b-1}}*(q_1\cdots q_I \, \sigma_{\lambda/H_b}) = q^{H_b} \, \sigma_{\lambda/H_b},
\end{align*}
where the second and third equalities follow from Proposition \ref{prop:expand},  Lemma \ref{lem:bijections},  \eqref{eq:q-Schubert}, and \eqref{eq:qpower}. 
\end{proof}

We can now prove Proposition \ref{pro:commutes}.
\begin{proof}[Proof of Proposition \ref{pro:commutes}]
Choose a vertex in the $(I+1)^{th}$ step of the ladder quiver (choosing this notation for compatibility with Theorem \ref{thm:thmB}), and let $z_v$ be the associated variable in the EHX mirror. We need to show that
\begin{equation} \label{eq:commutes} \pi(F_{\Gr}(\phi_{\Gr}(z_v)))=F_{\Fl}(\phi_{\Fl}(z_v)).\end{equation}
Suppose $v$ is in the $j^{th}$ row of the $k^{th}$ column of the $(I+1)^{th}$ block of the ladder quiver. Then
\[\phi_{\Fl}(z_v)=q_1 \dots q_{I} \frac{p^{I+1}_{j \times k}}{p^{I+1}_{(j-1) \times (k-1)}}.\]
As in the definition of the Schubert map, for $\ell=1,\dots,{I}$, let $R_\ell$ be the 
\[c \times (r_{\ell-1}-r_{\ell}), \hspace{2mm} c:=r_I-r_{I+1}+j-k\]
rectangle, and set $R_{I+1}:= j \times k$. Set $\overline{R}_{I+1}:= (j-1) \times (k-1)$.  
Set $\vec{\mu}_a:=(R_1,\dots,R_{I+1},\emptyset,\dots,\emptyset)$, $\vec{\mu}_b:=(R_1,\dots,\overline{R}_{I+1},\emptyset,\dots,\emptyset)$, and  $\vec{\mu}_2:=(R_1,\dots,R_{I},\emptyset,\emptyset,\dots,\emptyset)$. 
Then the right hand side of \eqref{eq:commutes} is 
\[q_1 \dots q_{I} \frac{\sigma_{\vec{\mu}_a}}{\sigma_{\vec{\mu}_2}}  \frac{\sigma_{\vec{\mu}_2}}{\sigma_{\vec{\mu}_b}}= q_1 \dots q_{I}\frac{\sigma_{\vec{\mu}_a}}{\sigma_{\vec{\mu}_b}}.\]
Next, we compute the left hand side of \eqref{eq:commutes}. It isn't hard to see that the vertex under consideration is in the $n-r_{I}+k$ column  and the $r_1-r_{I+1}+j$ row of the Grassmannian quiver. Let $\lambda_a=  (r_1-r_{I+1}+j) \times (n-r_{I}+k) $, and let $\lambda_b=   (r_1-r_{I+1}+j-1)\times (n-r_{I}+k-1)$. The left hand side of $\eqref{eq:commutes}$ is therefore
\[ \frac{s^1_{\lambda_a}}{s^1_{\lambda_b}}.\]
Both $\lambda_a$ and $\lambda_b$ are compatible with a $q$-hook. The partition $\lambda_a$ is compatible with the $q$-hook $H_{n-r_I+k}$. Note that we can partition $\lambda_a$ as in Figure \ref{fig:comparehook} (so, in the notation of Theorem \ref{thm:thmB}, $\overline{b}=k$):
\begin{figure}[h!]

\centering
\begin{tikzpicture}[scale=.5]
\fill[color=gray!40]   (0,0) --(9, 0) -- (9,-5) -- (2, -5) -- (2, -7)-- (0, -7) -- cycle;

\node[scale=.8] at (2,-1.5) {$R_{n-r_I+k}$};

\draw[thick] (6, 0) -- (9, 0) -- (9,-5) -- (2, -5) -- (2, -7) -- (0, -7)--(0,-3) -- (0,0) -- (6,0);
\draw (7,0)--(7,-3)--(0,-3);
\draw[thick] (0,-7) --  (0,-9) --(9,-9)--(9,-5);

\draw (2,-6) -- (2,-9);
\draw (4.67,-5) -- (4.67,-9);
\node[scale=.5] at (5, -6.1) {$\cdots$};
\draw (5.3,-5) -- (5.3,-9);
\draw (7,-5) -- (7,-9);
\node[scale=.8]  at (1, -8) {$R^{I+1}$};

\node[scale=.8]  at (3.3,-6.1) {$R^I$};

\node[scale=.8] at (6.1,-6.1) {$R^2$};
\node[scale=.8] at (7.6,-6.1) {$R^1$};

\node[scale=.6] at (1, -9.5) {$k$};
\draw [-|] (1.3, -9.5) -- (2, -9.5);
\draw [|-] (0, -9.5) -- (.7, -9.5);

\node[scale=.8] at (4.5, .5) {$n-r_I+k$};
\draw [-|] (6, .5) -- (9, .5);
\draw [|-] (0, .5) -- (3, .5);

\node[scale=.6, rotate = 90] at (-.5, -5) {$r_I-r_{I+1}$};
\draw [-|] (-.5, -6) -- (-.5, -7);
\draw [|-] (-.5, -3) -- (-.5, -4);

\node[scale=.6, rotate = 90] at (-.5, -1.5) {$r_1-r_I$};
\draw [-|] (-.5, -2.5) -- (-.5, -3);
\draw [|-] (-.5, 0) -- (-.5, -0.5);

\node[scale=.6, rotate = 90] at (-.5, -8) {$j$};
\draw [-|] (-.5, -8.5) -- (-.5, -9);
\draw [|-] (-.5, -7) -- (-.5, -7.5);

\node[scale=.6, rotate = 90] at (9.5, -2.5) {$r_1-r_I+k$};
\draw [-|] (9.5, -4) -- (9.5, -5);
\draw [|-] (9.5, 0) -- (9.5, -1);

\node[scale=.6, rotate = 90] at (9.5, -7) {$c$};
\draw [-|] (9.5, -7.5) -- (9.5, -9);
\draw [|-] (9.5, -5) -- (9.5, -6.5);

\end{tikzpicture}
\caption{The partitioning of $\lambda_a$.  }
\label{fig:comparehook}
\end{figure}
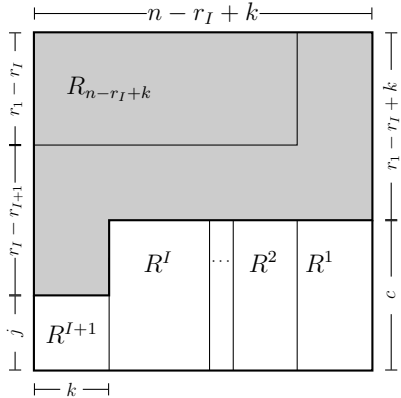
Therefore
\[ s^1_{\lambda_a}=q^{H_{n-r_I+k}} \sigma_{\vec{\mu}_1}.\]
Similarly, $\lambda_b$ is compatible with the $q$-hook $H_{n-r_I+k-1}$, and
\[ s^1_{\lambda_b}=q^{H_{n-r_I+k-1}} \sigma_{\vec{\mu}_b}.\]
It finally follows by comparing the $q$ factors that
\[ \frac{s^1_{\lambda_a}}{s^1_{\lambda_b}}=q_1 \dots q_{I}\frac{\sigma_{\vec{\mu}_a}}{\sigma_{\vec{\mu}_b}}\]
as required.

\end{proof}

We now prove the second part of Theorem \ref{thm:thmB}, that for $\lambda\subseteq r_1\times n$ such that $R_{\lambda_1}\subseteq \lambda$, but $H_{\lambda_1}\not\subseteq \lambda$, $s^1_\lambda =0$ in $\QH^*F$.

\begin{proof}[Proof of part (b) of Theorem \ref{thm:thmB}]
We proceed by induction on the width of $\lambda$. If $\lambda_1\leq n-r_1$, then $\lambda$ is compatible with a $q$-hook since $H_{\lambda_1}=\emptyset$, so the proposition holds vacuously.  
Now assume the result holds for partitions of width $b-1$, and consider a partition $\lambda$ of width $b:=\lambda_1$ that contains $R_b$ but not $H_b$. Let $I$ be such that $n-r_I<\lambda_1\leq n-r_{I+1}$.

We use the expansion  \eqref{eq:expansion0} of the determinant $s^1_\lambda$.   From Remark \ref{rem:qhook} and Definition \ref{def:qlambda}, $\lambda^{(m)}$ contains $R_{b-1}$ for $1\leq m\leq \lambda_1$. If $\lambda_1=n-r_{I+1}$, then $H_b=R_b$ and  there is nothing to prove, so assume $n-r_I<\lambda_1< n-r_{I+1}$ and write $\bar{b}=\lambda_1-(n-r_I)$.

First consider the case where $\bar{b}>1$ and $\lambda'_{\bar{b}-1}< r_1-r_{I+1}-1$. This corresponds to the cell marked $\otimes$ in Figure \ref{fig:lambdamhook} not being contained in $\lambda$. In this case, $(\lambda^{(m)})'_{\bar{b}-1}<r_1-r_{I+1}$ for $1\leq m\leq \lambda_1$ by Remark \ref{rem:qhook} so that $\lambda^{(m)}$ does not contain $H_{b-1}$. Then by the inductive hypothesis, $s^1_{\lambda^{(m)}} =0$. Since all the summands in \eqref{eq:expansion0} are zero, $s^1_\lambda=0$.

Now if $\bar{b}>1$ and  $\lambda'_{\bar{b}-1}\geq  r_1-r_{I+1}-1$, i.e. $\lambda$ contains the cell marked $\otimes$ in Figure \ref{fig:lambdamhook}, then by Remark \ref{rem:qhook} and Definition \ref{def:qlambda}, 
 if $m<\bar{b}$, then $\lambda^{(m)}$ does not contain $H_{b-1}$, so by the inductive hypothesis, 
 $s^1_{\lambda^{(m)}}=0$. On the other hand, 
 if $m\geq \bar{b}$, then  $\lambda^{(m)}$ contains $H_{b-1}$, and so by part (a) of Theorem \ref{thm:thmB}, 
 the expansion \eqref{eq:expansion0}   becomes
\[
s^1_\lambda
 =\sum_{m=\bar{b}}^b (-1)^{m-1} s^1_{1^{\lambda'_m-m+1}} *    q^{H_{b-1}}\sigma_{\lambda^{(m)}/H_{b-1}}.
\]
Since $H_{b-1}\not\subseteq\lambda^{(m)}$ for $m<\bar{b}$, by Remark \ref{rem:skewzero} and \eqref{eq:q-Schubert}, we have
\begin{equation} \label{eq:expansionm}
s^1_\lambda
 =\sum_{m=1}^b (-1)^{m-1} s^1_{1^{\lambda'_m-m+1}} *    q^{H_{b-1}}\Delta_{\lambda^{(m)}/H_{b-1}}(e^q(\psi)),
 \end{equation} 
 where $\psi=((i+1)^{\bar{b}-1}, i^{r_{I-1}-r_I},\ldots,1^{n-r_1})$.  
 Similarly, if $\bar{b}=1$, then $H_{b-1}=R_{b-1}$ and $\lambda^{(m)}$ contains $H_{b-1}$ for all $1\leq m\leq b$, so by part (a) of Theorem \ref{thm:thmB} and \eqref{eq:q-Schubert}, we have \eqref{eq:expansionm}  as well. Moreover, by Proposition \ref{prop:expand},  \eqref{eq:expansionm} becomes
 \[ s^1_\lambda= q_1\cdots q_I \, \Delta_{\lambda/H_b}(e^q(\phi)) \text{ in } \QH^*\flag,\]
where  $\phi=((i+1)^{\bar{b}},i^{r_{I-1}-r_I},\ldots,1^{n-r_1})$. Since $H_b\not\subseteq \lambda$,  we conclude that $s^1_\lambda=q_1\cdots q_I \cdot 0=0$ by Remark \ref{rem:skewzero}.
\end{proof}

\bibliographystyle{amsplain}
\bibliography{bibliography}
\end{document}